\documentclass[11pt,reqno]{amsart}

\usepackage{amsthm,amssymb,mathtools}
\usepackage[pagebackref,colorlinks,linkcolor=red,citecolor=blue,urlcolor=blue,hypertexnames=false]{hyperref}
\usepackage{amsrefs}
\usepackage{amscd}
\usepackage[utf8]{inputenc}
\usepackage{tikz}
\usetikzlibrary{matrix,arrows}
\usepackage[all]{xy}

\usepackage{enumerate}

\def\cC{\mathcal C} 
\def\cD{\mathcal D}

\def\fF{\mathfrak{F}}

\def\cF{\mathcal F}
\def\cE{\mathcal E}
\def\cT{\mathcal T}

\def\cO{\mathcal O}

\def\cS{\mathcal S}

\newcommand{\U}{\mathbb{U}}
\newcommand{\R}{\mathbb{R}}

\newcommand{\fin}{{\mathcal{F}in}}
\newcommand{\vcyc}{\mathcal{V}cyc}

\newcommand{\Z}{\mathbb{Z}}
\newcommand{\N}{\mathbb{N}}
\newcommand{\Dinf}{{D_\infty}}
\newcommand{\Ht}{H^{\langle t \rangle}}
\newcommand{\ltr}{{\langle t \rangle}}

\newcommand{\balpha}{\overline{\alpha}}
\newcommand{\balpham}{\overline{\alpha^{-1}}}
\newcommand{\balphapm}{\overline{\alpha^{\pm1}}}
\newcommand{\odd}{\mathrm{o}}
\newcommand{\even}{\mathrm{e}}
\renewcommand{\S}{S}

\newcommand{\So}{S_\odd}
\newcommand{\Se}{S_\even}
\newcommand{\Sop}{S_{\odd^+}}
\newcommand{\psioe}{\psi_{\mathrm{oe}}}
\newcommand{\psieo}{\psi_{\mathrm{eo}^+}}

\newcommand{\Hom}{\mathrm{Hom}}

\newcommand{\bfK}{\textbf{K}}
\newcommand{\pt}{\mathrm{pt}}
\newcommand{\OrDinf}{{\mathrm{Or}\Dinf}}
\newcommand{\Ch}{\mathrm{Ch}}
\newcommand{\Ab}{\mathrm{Ab}}
\newcommand{\op}{\mathrm{op}}

\newcommand{\OGR}{\widetilde{\cO}^{\langle t \rangle}(\R)}

\newcommand{\ODR}{\cO^{\Dinf}(\overline{\R})}

\def\supp{\operatorname{supp}}

\def\id{\operatornamewithlimits{id}}

\renewcommand{\Im}{\mathrm{Im}}
\newcommand{\assem}{\mathrm{assem}}

\newcommand{\size}{\mathrm{size}}
\newcommand{\hsi}{\mathrm{hsize}}
\newcommand{\vsi}{\mathrm{vsize}}
\renewcommand{\id}{\mathrm{id}}

\numberwithin{equation}{section}
\theoremstyle{plain}
\newtheorem{thm}[equation]{Theorem}

\newtheorem{cor}[equation]{Corollary}
\newtheorem{lem}[equation]{Lemma}

\newtheorem{prop}[equation]{Proposition}

\theoremstyle{definition}
\newtheorem{defn}[equation]{Definition}
\theoremstyle{remark}
\newtheorem{rem}[equation]{Remark}

\begin{document}
\title[Two examples of vanishing and squeezing in $K_1$]{Two examples of vanishing \\ and squeezing in $K_1$}
\author{E. Ellis}

\email{eellis@fing.edu.uy}
\address{IMERL. Facultad de Ingenier\'\i a. Universidad de la Rep\'ublica. Montevideo, Uruguay.}
\author{E. Rodr\'iguez Cirone}

\email{ercirone@dm.uba.ar}
\address{Dep. Matemática -- FCEyN -- UBA, Buenos Aires, Argentina.}
\author{G. Tartaglia}
\email{gtartaglia@mate.unlp.edu.ar}
\address{Dep. Matem\'atica-CMaLP, FCE-UNLP, La Plata, Argentina.}
\author{S. Vega}
\email{svega@dm.uba.ar}
\address{Dep. Matemática -- FCEyN -- UBA, IMAS -- CONICET , Buenos Aires, Argentina.}
\thanks{All authors were partially supported by grant ANII FCE-3-2018-1-148588. The first author is partially supported by ANII, CSIC and PEDECIBA. G. Tartaglia and S. Vega were supported by CONICET . The last three authors were partially supported by grants UBACYT 20020170100256BA and PICT 2017--1935}

\maketitle

\begin{abstract}
    
 Controlled topology is one of the main tools for proving the isomorphism conjecture concerning the 
    algebraic $K$-theory of group rings.
    In this article we dive into this machinery in two examples: when the group is infinite cyclic and when it is the infinite dihedral group --- in both cases with the family of finite subgroups. We prove a vanishing theorem and show how to explicitly squeeze the generators of these groups in $K_1$. For the infinite cyclic group, when taking coefficients in a regular ring, we get a squeezing result for every element of $K_1$; this follows from the well-known result of Bass, Heller and Swan.
\end{abstract}


\section{Introduction}
Let $G$ be a group, $\cF$ a family of subgroups of $G$, $R$ a ring and $\mathbf{K}$ the non-connective algebraic $K$-theory spectrum. The isomorphism conjecture for $(G,\cF,R,\mathbf{K})$ identifies the algebraic $K$-theory of the group ring $RG$ with an equivariant homology theory evaluated on $E_{\cF}G$, the universal $G$-$CW$-complex with isotropy in $\cF$. More precisely, the conjecture asserts that the following assembly map --- induced by the projection of $E_{\cF}G$ to the one-point space $G/G$ --- is an isomorphism \cite{dl}: 
\begin{equation}\label{as}
\text{assem}_{\cF}: H^G_*(E_{\cF}G,\textbf{K}(R)) \to H^G_*(G/G,\textbf{K}(R)) \cong K_*(RG)
\end{equation}
The left hand side of \eqref{as} provides homological tools which may facilitate the computation of the $K$-groups.
 
For $\cF=\vcyc$, the family of virtually cyclic subgroups, the conjecture is known as the Farrell-Jones conjecture \cite{fj},\cite{bfjr}. Although this conjecture is still open,
it is known to hold for a large class of groups, among which are hyperbolic groups \cite{blr}, CAT(0)-groups \cite{bl}, solvable groups \cite{w} and mapping class groups \cite{bb}.  One of the main methods of proof is based on controlled topology, and its key ingredient is an \emph{obstruction category} whose $K$-theory coincides with the homotopy fiber of the assembly map.

For a free $G$-space $X$, the objects of the obstruction category $\cO^G(X)$ are $G$-invariant families of finitely generated free $R$-modules $\{M_{(x,t)}\}_{(x,t)\in X\times[1,\infty)}$ whose support is a locally finite subspace of $X\times [1,\infty)$. A morphism in $\cO^G(X)$ is a $G$-invariant family of $R$-module homomorphisms satisfying the continuous control condition at infinity.
Associated to $\cO^G(X)$ there is a Karoubi filtration 
\[\mathcal{T}^{G}(X)\rightarrow \mathcal{O}^{G}(X)\rightarrow \mathcal{D}^{G}(X)\]
that induces a long exact sequence in $K$-theory:
\[
\small{\ldots} \rightarrow K_{*+1}( \mathcal{O}^{G}(X))\rightarrow  K_{*+1}( \mathcal{D}^{G}(X)) \xrightarrow{\partial} K_{*}( \mathcal{T}^{G}(X))\rightarrow K_{*}( \mathcal{O}^{G}(X))\rightarrow \small{\ldots}
\]
The previous definitions can be generalized for non-necessarily free $G$-spaces. Taking $X=E_{\cF}G$,
the assembly map \eqref{as} identifies with the connecting homomorphism $\partial$ of the above sequence.
Hence, an element $[\alpha] \in K_{*} (\mathcal{T}^{G}(X))  $ belongs to the image of the assembly map if and only if this element vanishes in $K_{*}( \mathcal{O}^{G}(X))$. 

If $X$ admits a $G$-invariant metric $d$, there is a notion of size for morphisms in $\cO^G(X)$. Given $\epsilon >0$, we say that $\varphi \in \cO^G(X)$ is \emph{$\epsilon$-controlled} over $X$ if $d(x,y)<\epsilon$, $\forall (x,t),(y,s)$ in the support of $\varphi$. If $\varphi$ is an $\epsilon$-controlled automorphism such that $\varphi^{-1}$ is also $\epsilon$-controlled, we call it an $\epsilon$-\emph{automorphism}.  The general strategy for proving that the obstruction category has trivial $K_1$ is the following: first show that there exists an $\epsilon>0$ such that $\epsilon$-automorphisms have trivial $K$-theory (\emph{vanishing result}), and then verify that every morphism has a representative in $K$-theory which is an $\epsilon$-automorphism (\emph{squeezing result}); see  \cite{ba}*{Corollary 4.3}, \cite{b}*{Theorem 2.10},  \cite{blr}, \cite{ped}*{Theorems 3.6 and 3.7}, \cite{rv}*{Theorem 37}.

In this article we examine how the previous machinery works in two examples: 
\begin{enumerate}[(i)]
    \item\label{ex:1} the infinite cyclic group $G=\langle t\rangle$ and the family $\cF$ consisting only of the trivial subgroup;
    \item\label{ex:2} the infinite dihedral group $G=\Dinf$ and the family $\cF=\fin$ of finite subgroups.
\end{enumerate}

In both cases, it is easily verified that $\R$ is a model for $E_\cF G$. In the first example, $t$ acts by translation by $1$. In the second one, we use the following presentation of the infinite dihedral group:
\begin{equation}\label{presDinf}\Dinf=\langle r, s\mid s^2=1, rs=sr^{-1}\rangle.\end{equation}
Then $r$ acts by translation by $1$ and $s$ acts by symmetry with respect to the origin. By the discussion above, in both examples, the assembly map in degree $1$ identifies with the morphism:

\begin{equation}\label{ass1}
\partial:K_2(\mathcal{D}^G(\mathbb{R}))\rightarrow K_1(\mathcal{T}^G(\mathbb{R})).
\end{equation}
Adapting ideas of Pedersen \cite{ped} to these $G$-equivariant settings, we prove the following vanishing result.
\begin{thm}[Theorem \ref{vani}]\label{intro:thmvani}
Let $G=\ltr$ or $G=\Dinf$. If $\alpha$ is a $\frac{1}{30}$-automorphism in $\mathcal{O}^G(\R)$, then $\alpha$ has trivial class in $K_1$.
\end{thm}
\noindent As an application, we get a sufficient condition for an element of $K_1(\mathcal{T}^G(\R))$ to be in the image of $\partial$.
\begin{cor}
Let $G=\ltr$ or $G=\Dinf$. If $\alpha$ is an $\frac{1}{30}$-automorphism in $\mathcal{T}^G(\R)$,  then $[\alpha]\in K_1(\mathcal{T}^G(\R))$ is in the image of the assembly map \eqref{ass1}.
\end{cor}
\noindent This illustrates the idea that \emph{small} automorphisms in $K_{1}(\mathcal{T}^G(\R))$ should belong to the image of $\partial$.

Let us now take a closer look at the image of \eqref{ass1}; we will focus on example \eqref{ex:1}.

A well-known theorem of Bass-Heller-Swan computes, for any ring $R$, the algebraic $K$-theory of the Laurent polynomial ring $R[t,t^{-1}]$ in terms of the $K$-theory of $R$. The group $K_0(R)\oplus K_1(R)$ is always a direct summand of $K_1(R[t,t^{-1}])$, and its inclusion is given by the following formula (see \cite{bhs}):
\[\xymatrix{
[M] \oplus ([M'],\tau)\ar@{|->}[r]^-{\psi}  & [R[t,t^{-1}]\otimes_{R} M,  t\otimes \operatorname{id}] + [R[t,t^{-1}]\otimes_{R} M', \operatorname{id}\otimes \tau]}.
\]
It can be shown that $\psi$ and the assembly map \eqref{ass1} fit into a commutative square as follows:
\begin{equation}\label{cuadradito}
\xymatrix@R=1.5em{
 K_{2}(\mathcal{D}^{\ltr}(\mathbb{R}))\ar[d]_{\cong}\ar[r]^-{\partial}& K_{1}(\mathcal{T}^{\ltr}(\mathbb{R}))\ar[d]^{\cong}_{\mathbb{U}}
\\
K_{0}(R)\oplus K_{1}(R)\ar[r]^{\psi} & K_{1}(R[t^{-1},t]).}
\end{equation}
Notice that $\partial$ takes into account the geometry of $\R$ while $\psi$ is purely algebraic. The isomorphism $\mathbb{U}$ is, moreover, induced by the functor that forgets geometry. Thus, we can regard the Bass-Heller-Swan morphism $\psi$ as the algebraic shadow of the assembly map.

It is clear from the formula above that $t$ belongs to the image of $\psi$. However, the obvious representation of $t$ as an automorphism in $\mathcal{T}^\ltr(\R)$ is not small --- indeed, it has size $1$. This phenomenon was already mentioned in \cite{b}*{Remark 2.14}. In our context we prove a squeezing result for $t$.
\begin{prop}[Proposition \ref{tchiq}]\label{introtchiq}
Let $n\in \N$. Then there exists an $\frac{1}{n}$-automorphism $\xi$ in $\mathcal{T}^\ltr(\R)$ such that $\mathbb{U}([\xi])=[t]$ in $K_1(R[t,t^{-1}])$.
\end{prop}
\noindent If we further assume that $R$ is regular, the latter result and the proof of 
\cite{bhs}*{Theorem 2} imply the following.
\begin{prop}[Proposition \ref{Achic}]
Let $R$ be a regular ring and let $\varepsilon>0$. For every $x\in K_{1}(R[t,t^{-1}])$ there exists an  $\epsilon$-automorphism $\xi$ in $\mathcal{T}^\ltr(\R)$ such that $\mathbb{U}([\xi])=x$.
\end{prop}

In example \eqref{ex:2}, it can be shown that both $r$ and $s$ belong to the image of the assembly map, and one may try to represent these elements by $\varepsilon$-automorphisms, for small $\varepsilon>0$. In the case of $r$, the proof of Proposition \ref{introtchiq} carries on verbatim to show that, for every $n\in\N$, there is an $\frac1n$-automorphism $\xi$ in $\mathcal{T}^\Dinf(\R)$ such that $\mathbb{U}([\xi])=[r]$. In the case of $s$, it is possible to find a $0$-automorphism representing this element (Remark. \ref{rem:s-chico}).

The rest of the paper is organized as follows. In section \ref{sec:generalset} we mainly fix notation and recall from \cite{bfjr} the basic definitions and results from controlled topology. In section \ref{sec:twoex} we study algebraically the assembly maps in the two examples mentioned above. In the case of example \eqref{ex:1}, we use Mayer-Vietoris to identify the domain of the assembly map \eqref{ass1} with $K_0(R)\oplus K_1(R)$. In the case of example \eqref{ex:2}, we use the equivariant Atiyah-Hirzebruch spectral sequence and \cite{ranicki}*{Corollary 3.27} to show that the assembly map is an isomorphism for regular $R$. Section \ref{sec:vanishing} contains the proof of Theorem \ref{intro:thmvani}. In section \ref{sec:squeezing} we discuss the notion of size in terms of matrices and we prove Proposition \ref{introtchiq}.

\subsection*{\it Acknowledgements} The authors wish to thank the organizers of the workshop {\it Matemáticas en el Cono Sur}, where this project was initiated, and Holger Reich, for his helpful comments. The last three authors also thank Eugenia Ellis for her hospitality and support during their visits to the IMERL--UdelaR in Montevideo.

\section{General setting}\label{sec:generalset}

\subsection{Geometric modules}
Let $R$ be a unital ring and $X$ a space. The additive category $\cC(X)=\cC(X;R)$ of geometric $R$-modules over $X$ is defined as follows. An object is a collection $A=(A_x)_{x\in X}$ of finitely generated free $R$-modules whose support $\supp(A)=\left\{x\in X : A_x\neq 0\right\}$ is locally finite in $X$. Recall that a subset $S\subset X$ is \emph{locally finite} if each point of $X$ has an open neighborhood whose intersection with $S$ is a finite set. A morphism $\varphi:A=(A_x)_{x\in X} \to B=(B_y)_{y\in X}$ consists of a collection of morphisms of $R$-modules $\varphi_x^y:A_x\to B_y$ such that the set $\left\{x : \varphi_x^y\neq 0\right\}$ is finite for every $y\in X$ and the set $\left\{y : \varphi_x^y\neq 0\right\}$ is finite for every $x\in X$. The support of $\varphi$ is the set
\[
\supp(\varphi)=\left\{(x,y)\in X\times X : \varphi_x^y\neq 0\right\}.
\]
Composition is given by matrix multiplication.

Let $G$ be a group which acts on $X$. Then there is an induced action on $\cC(X)$ given by $(g^*A)_x= A_{gx}$ and $(g^*\varphi)_x^y=\varphi_{gx}^{gy}$. A geometric $R$-module $(A_x)_{x\in X}$ is called $G$-invariant if $g^*A=A$. A morphism $\varphi$ between $G$-invariant geometric $R$-modules is called $G$-invariant if $g^*\varphi=\varphi$. The category of $G$-invariant geometric $R$-modules and $G$-invariant morphisms is denoted $\cC^G(X)$. It is an additive subcategory of $\cC(X)$. 

\subsection{Restriction to subspaces}

Let $A$ be a geometric $R$-module on $X$ and let $Y\subseteq X$ be a subspace. We will write $A|_Y$ for the geometric module over $X$ defined by
\[\left(A|_Y\right)_x = \left\{\begin{array}{cl}A_x & \mbox{if $x\in Y$,}\\ 0 & \mbox{otherwise.}\end{array}\right.\]
Notice that $A|_Y$ is a submodule of $A$.

\begin{rem}
If $Y=\cup_iY_i$ is a disjoint union, then $A|_Y=\oplus_iA|_{Y_i}$.
\end{rem}

Let $A$ and $B$ be geometric modules over $X$, let $Y,Z\subseteq X$ be subspaces, and let $\alpha:A\to B$ be a morphism. The decompositions $A=A|_Y\oplus A|_{Y^c}$ and $B=B|_Z\oplus B|_{Z^c}$ induce a matrix representation
\[\alpha=\left(\begin{array}{cc}
    \alpha|_Y^Z & \alpha|_{Y^c}^Z \\[0.5em]
    \alpha|_Y^{Z^c} & \alpha|_{Y^c}^{Z^c}
\end{array}\right)\]
where $\alpha|_V^W:A|_V\to B|_W$, $V\in\{Y, Y^c\}$, $W\in\{Z, Z^c\}$. This gives a well defined function:
\[?|_Y^Z:\Hom_{\cC(X)}(A,B)\to \Hom_{\cC(X)}(A|_Y, B|_Z)\]
The following properties are easily verified:
\begin{enumerate}
    \item If $\alpha, \beta:A\to B$, then $(\alpha+\beta)|_Y^Z=\alpha|_Y^Z+\beta|_Y^Z.$
    \item If $Y=Y_1\cup\cdots \cup Y_m$ and $Z=Z_1\cup \cdots\cup Z_n$ are disjoint unions, then the decompositions $A|_Y=A|_{Y_1}\oplus\cdots\oplus A|_{Y_m}$ and $B|_Z=B|_{Z_1}\oplus\cdots\oplus B|_{Z_n}$ induce the matrix representation:
    \[\alpha|_Y^Z=\begin{pmatrix}\alpha|_{Y_1}^{Z_1} & \cdots & \alpha|_{Y_m}^{Z_1}\\
    \vdots & \ddots & \vdots \\
    \alpha|_{Y_1}^{Z_n} & \cdots & \alpha|_{Y_m}^{Z_n}\end{pmatrix}\]
    \item If $\alpha:A\to B$, $\beta:B\to C$ and $X=X_1\cup \cdots \cup X_n$ is a disjoint union, then:
    \[\left.\left(\beta\circ\alpha\right)\right|_Y^Z=\sum_{i=1}^n\beta|_{X_i}^Z\circ \alpha|^{X_i}_Y\]
\end{enumerate}

\begin{rem}
The above definitions make sense in the equivariant setting. If $X$ is a $G$-space, $Y\subseteq X$ is a $G$-invariant subspace and $A$ is a $G$-invariant geometric module, then $A|_Y$ is $G$-invariant as well. If $\alpha$ is a $G$-invariant morphism and $Y,Z\subseteq X$ are $G$-invariant subspaces, then $\alpha|_Y^Z$ is $G$-invariant as well.
\end{rem}

\begin{defn}
Let $\gamma:A\to B$ be a morphism of geometric modules and let $Y\subseteq X$ be a subspace. We say that $\gamma$ \emph{is zero on $Y$} if the decompositions $A=A|_Y\oplus A|_{Y^c}$ and $B=B|_Y\oplus B|_{Y^c}$ induce a matrix representation:
\[\gamma=\begin{pmatrix}0 & 0 \\ 0 & \ast\end{pmatrix}\]
\end{defn}

\begin{defn}
Let $A$ be a geometric module on $X$, let $\gamma:A\to A$ be an endomorphism and let $Y\subseteq X$ be a subspace. We say that $\gamma$ \emph{is the identity on $Y$} if the decomposition $A=A|_Y\oplus A|_{Y^c}$ induces a matrix representation:
\[\gamma=\begin{pmatrix}\id & 0 \\ 0 & \ast\end{pmatrix}\]
\end{defn}

\begin{rem}
It is easily verified that if $\gamma$ is the identity on $Y$ (respectively, zero on $Y$) and $Z\subseteq Y$, then $\gamma$ is the identity on $Z$ (resp. zero on $Z$).
\end{rem}

\begin{rem}
Let $\alpha:A\to B$ and $\beta:B\to C$ be morphisms of geometric modules on $X$ and let $Y,Z\subseteq X$ be subspaces. If $\alpha$ is the identity on $Y$, then $(\beta\circ\alpha)|_Y^Z=\beta|_Y^Z$. Indeed,
\begin{align*}
    (\beta\circ\alpha)|_Y^Z&=\beta|^Z_Y\circ \alpha|_Y^Y+\beta|_{Y^c}^Z\circ\alpha|_Y^{Y^c}\\
    &=\beta|_Y^Z\circ \id + \beta|_{Y^c}^Z\circ 0=\beta|_Y^Z.
\end{align*}
In the same vein, if $\beta$ is the identity on $Z$, then $(\beta\circ\alpha)|_Y^Z=\alpha|_Y^Z$. Also, if $\alpha$ is zero on $Y$ or $\beta$ is zero on $Z$ then $(\beta\circ\alpha)|_Y^Z=0$. We will use these properties in Section \ref{sec:vanishing} without further mention.
\end{rem}

\subsection{Control conditions.}  Let $X$ be a $G$-space and equip $X\times [1,\infty)$ with the diagonal action, where $G$ acts trivially on $[1,\infty)$. We need to impose some support conditions on objects and morphisms of $\cC^G(X)$ and $\cC^G(X\times[1,\infty))$
\smallskip

\emph{Object support condition} Let $\cS_{Gc}^X$ be the set of $G$-compact subsets of $X$; i.e. the set of all subsets of the form $GK$, with $K\subset X$ compact.

\emph{Morphism support condition} Let $G_x$ be the stabilizer subgroup of $x\in X$, and write $\cE_{Gcc}^X$ for the collection of all subsets $E\subset (X\times [1,\infty))\times (X\times [1,\infty))$ satisfying:

\begin{enumerate}
\item For every $x\in X$ and for every $G_x$-invariant open neighborhood $U$ of $(x,\infty)$ in $X\times [1,\infty]$, there exists a $G_x$-invariant  open neighborhood $V\subset U$ of $(x,\infty)$ in $X\times [1,\infty]$ such that
\[ ((X\times[1,\infty] - U)\times V)\cap E=\emptyset;
\]
\item the projection $(X\times[1,\infty))^{\times 2}\to [1,\infty)^{\times 2}$ sends $E$ into a subset of the form $\{(t,t')\in[1,\infty)\times[1,\infty):|t-t'|\leq \delta\}$ for some $\delta<\infty$;
\item $E$ is symmetric and invariant under the diagonal $G$-action.
\end{enumerate}

The collection $\cE_{Gcc}^X$ is called the \emph {equivariant continuous control condition}.

\begin{lem}(c.f. \cite{bfjr}*{Lemma 2.10}) \label{lem equiv}If $X$ is a free $G$-space, then $\cC^G(X;\cS^X_{Gc})$ is equivalent to the category $\fF_{RG}$ of finitely generated free $RG$-modules. 

\end{lem}

Given a free $G$-space $X$, we will write
\[\U:\cC^G(X;\cS^X_{Gc})=\cC^G(X;\cS^X_{Gc},R)\to \cC(\pt,RG)=\fF_{RG}\]
for the functor that induces the equivalence of the previous Lemma:
\begin{align*}
\mathbb{U}(A)&= \bigoplus_{x \in X}A_x\\
\mathbb{U}(\varphi:A\to B)&=\bigoplus_{(x,y)\in\supp(\varphi)}\varphi_x^y:\bigoplus_{x \in X}A_x \to \bigoplus_{y \in X}B_y
\end{align*}

 Note that, if we fix a basis for every finitely generated free $R$-module,   $\mathbb{U}(\varphi)$ is a matrix indexed by $\supp(\varphi)$ such that each entry is a finite matrix $[\varphi_x^y]$ with coefficients in $R$. Using the $G$-invariance property of objects and morphisms in $\cC^G(X;\cS_{Gc})$, we will interpret $\mathbb{U}(\varphi)$ as a finite matrix with coefficients in $RG$. For $S$ a complete set of representatives of $G\backslash X$, we will abuse notation and write
\[\mathbb{U}(\varphi)_{(s,t)}=\sum_{g\in G} g[\varphi^{gs}_t], \hspace{0.4cm} \forall (s,t)\in (S\times S)\cap \supp(\varphi).\]
 The locally finite and $G$-compact conditions for the support of objects in $\cC^G(X;\cS^X_{Gc})$ guarantees that $|(S\times S)\cap \supp(\varphi)|<\infty.$
\subsection{Resolutions} The construction in the previous lemma allows us to identify a finitely generated free $RG$-module with a geometric $R$-module over a \emph{free} $G$-space $X$. Following \cite{bfjr}, this restriction is avoided introducing \emph{resolutions}.

\begin{defn}(c.f. \cite{bfjr}*{Section 3})
Given a $G$-space $X$, a \emph{resolution} of $X$ is a free $G$-space $\overline{X}$ together with a continuous $G$-map $p:\overline{X}\to X$ satisfying the following conditions:
\begin{itemize}
\item the action of $G$ on $\overline{X}$ is properly discontinuous and the orbit space $G\backslash X$ is Hausdorff (this is always the case if $\overline{X}$ is a $G$-CW-complex);
\item for every $G$-compact set $GK\subset X$ there exists a $G$-compact set $G\overline{K}\subset \overline{X}$ such that $p(G\overline{K})=GK$.
\end{itemize} 
\end{defn}

\begin{rem}
The projection  $X \times G \to X$ is always a resolution of $X$ called the \emph{standard resolution}.
\end{rem}

Let $p:\overline{X}\to X$ be a resolution of a $G$-space $X$, and let $\pi:\overline{X}\times[1,\infty)\to \overline{X}$ be the projection. We will abuse notation and set

\[(p\times \text{id})^{-1}(\cE^X_{Gcc})=\left\{((p\times \text{id})^2)^{-1}(E) : E\in \cE^X_{Gcc} \right\}, \]
\[\pi^{-1}(\cS^{\overline{X}}_{Gc})=\left\{ \pi^{-1}(S): S \in \cS^{\overline{X}}_{Gc}\right\}.\]

\smallskip

We define the following additive subcategories of $\cC^G(\overline{X})$ and $\cC^G(\overline{X} \times [1,\infty))$:

\begin{itemize}
\item $\cT^G(\overline{X})= \cC^G(\overline{X};\cS_{Gc}^{\overline{X}})$, the category of $G$-invariant geometric $R$-modules over $\overline{X}$ whose support is contained in a $G$-compact subset of $\overline{X}$.

\item $\cO^G(\overline{X})= \cC^G(\overline{X}\times[1,\infty);(p\times \text{id})^{-1}(\cE^X_{Gcc}), \pi^{-1}(\cS^{\overline{X}}_{Gc}))$, the category whose objects are $G$-invariant geometric $R$-modules over $\overline{X}\times[1,\infty)$ with support contained in some $S\in \pi^{-1}(\cS^{\overline{X}}_{Gc})$, and whose morphisms are the $G$-invariant $R$-module morphisms with support contained in some $E\in (p\times \text{id})^{-1}(\cE^X_{Gcc})$.

\item $\cD^G(\overline{X})= \cO^G(\overline{X})^\infty$, the category of germs at infinity. It has the same objects as $\cO^G(\overline{X})$, but morphisms are identified if their difference can be factored over an object whose support is contained in $\overline{X}\times [1,r]$ for some $r \in [1,\infty)$.
\end{itemize}

When considering the standard resolution $G\times X\to X$ we write $\cT^G(X)=\cT^G(G\times X)$, $\cO^G(X)=\cO^G(G\times X)$ and $\cD^G(X)=\cD^G(G\times X)$.
\begin{thm}(c.f. \cite{bfjr}*{Proposition 3.5}). Let $X$ be a $G$-space and $p:\overline{X}\to X$ and $p':\overline{X}'\to X$ be two resolutions of $X$. Then the germ categories $\cD^G(\overline{X})$ and $\cD^G(\overline{X}')$ are equivalent.
\end{thm}

\begin{lem}\label{times-libre}
Let $X$ and $Y$ be $G$-spaces and suppose the action on $X$ is free. Write  $Y^\tau$ for the space $Y$ with trivial $G$-action. Then $Y \times X$ is isomorphic to $Y^\tau\times X$ as a $G$-space. 
\end{lem}

\begin{proof}  Let $\rho:X\to G\backslash X$ be the projection to the orbit space and $s: G\backslash X \to X$  a section. If $x\in X$, write $h_x$ for the unique element of $G$ that verifies $x=h_xs(\rho(x))$. Define $\varphi:Y\times X \to Y^\tau\times X$ by $\varphi (y,x)= (h^{-1}_x y,x)$. It is easy to see that $\varphi$ is an isomorphism of $G$-spaces. 
\end{proof}

\begin{lem}\label{restriccion-libre}
Let $X$ be $G$-space such that there exists a subgroup
$K \leq G$ which acts freely on $X$ with the restricted action from $G$. Then $G/K\times X$ (with diagonal action) is a free $G$-space.
\end{lem}
\begin{proof}
Suppose $g(hK,x) = (hK,x)$. Then $ghK=hK$ and there exists $k \in K$ with $g=hkh^{-1}$. Then $hkh^{-1}x=x$ and $k(h^{-1}x)=h^{-1}x$.  But since the action of $K$ over $X$ is free, this implies that $k=1_G$ and then $g=1_G$. This concludes that $G/K\times X$ is free.
\end{proof}

\begin{lem}\label{times-discreto}
Let $X$ be a $G$-space and $Y$ a discrete space with trivial $G$-action. Then $\mathcal{C}^G(Y \times X;\cS_{Gc}^{Y\times X})$ is equivalent to
$\mathcal{C}^G(X;\cS_{Gc}^X)$.
\end{lem}

\begin{proof}
Let $A$ be an object of $\cC^G(Y\times X; \cS_{Gc}^{Y\times X})$, then $\supp(A)$ is locally finite and contained in a $G$-compact subset of $Y\times X$. Hence, there exist $y_1,y_2,...y_n \in Y$ and $G$-compact subsets $F_1, F_2,\ldots,F_n$ of $X$ such that $\supp(A)\subseteq \bigcup_{i=1}^{n} \{y_i\}\times F_i$.  Consider the direct sum of additive categories $\bigoplus_{y\in Y} \cC^G(X;\cS_{Gc}^X)$  and define the functor 
\begin{gather*}
F:\cC^G(Y\times X; \cS_{Gc}^{Y\times X}) \to \bigoplus_{y\in Y} \cC^G(X;\cS_{Gc}^X) \text{ by} \\
F(A)=\{(A_{(y,x)})_{x\in X}\}_{y\in Y},\\
F(\varphi:A\to B)=\{ (\varphi_{(y_1,x_1)}^{(y_2,x_2)})_{x_1,x_2\in X}:F(A)_{y_1} \to F(B)_{y_2}\}_{y_1,y_2 \in Y}.
\end{gather*}
 Clearly $F$ induces an equivalence of categories. Moreover, sending an element of $\bigoplus_{y\in Y} \cC^G(X;\cS_{Gc}^X)$ to the direct sum of its components induces an equivalence $\bigoplus_{y\in Y} \cC^G(X;\cS_{Gc}^X) \cong  \cC^G(X;\cS_{Gc}^X)$.
 \end{proof}

\begin{rem}\label{times-control}
Note that in the previous lemma we haven't specified any control conditions on the morphisms.  In the isomorphism
\[ \bigoplus_{y \in Y} \cC^G(X \times [1,\infty);\cE_{Gcc}^X, \pi^{-1}(\cS_{Gc}^X)) \cong  \cC^G(X\times [1,\infty); \cE_{Gcc}^X,\pi^{-1}(\cS_{Gc}^X)), \]
multiple copies of the direct sum may be used but no condition on morphisms 
--- other than finiteness in rows and columns and $G$-equivariance, which the direct sum already
satisfies --- are enforced; this means that in order to have an isomorphism
\[ \bigoplus_{y \in Y} \cC^G(X \times [1,\infty);\cE_{Gcc}^X, \pi^{-1}(\cS_{Gc}^X))\cong 
\cC^G(Y\times X \times [1,\infty); \cE,\pi^{-1}(\cS_{Gc}^{X\times Y})), \]
the control condition $\cE$ isn't given by $\cE_{Gcc}^{Y\times X}$. Instead,
if $p:Y\times X\to X$ is the projection, then 
\[ \cE = (p\times id)^{-1}(\cE_{Gcc}^X).\]
\end{rem}

 \begin{cor}\label{tilde} If $X$ is a free $G$-space we have the following equivalences of categories:
 \begin{align*}
 \cT^G(X)&\cong \widetilde{\cT}^G(X):=\cC^G(X,\cS_{Gc}^X) ;\\
 \cO^G(X)&\cong \widetilde{\cO}^G(X):= \cC^G(X\times [1,\infty),\cE_{Gcc}^X, \pi^{-1}(\cS_{Gc}^X));\\
 \cD^G(X)&\cong \widetilde{\cD}^G(X):=\tilde{\cO}^G(X)^\infty;
 \end{align*}
 where  $\pi:X\times[1,\infty)\to X$ is the projection. 
\end{cor}

\begin{proof}
Take $Y=G$ in Lemmas \ref{times-libre} and \ref{times-discreto}.
\end{proof}

\begin{defn}\label{def sz}
Let $X$ be a $G$-space equipped with a $G$-invariant metric $d_G$. Let $p:\overline{X}\to X$ be a
resolution and $\varphi \in \cT^G(\overline{X})$. We define the \emph{size} of
$\varphi$ as the supremum of the distances between the components of $\varphi$ when
projected to $X$:
\[ \size(\varphi) = \sup \{ d_G(p(\overline{x}),p(\overline{y})) :
(\overline{x},\overline{y}) \in \supp{\varphi}\}\]

We also extend this size to morphisms in $\cO^G(\overline{X})$.
In this case we define the \emph{horizontal size} of $\varphi$ by measuring
the distance in $X$:
\[\hsi(\varphi):= \sup \{d_G(p(\bar{x}),p(\bar{y})) : (\bar{x},t,\bar{y},s)\in \supp(\varphi)\}\]
If we view $\cT^G(\overline{X})$ as a subcategory of $\cO^G(\overline{X})$ then we have that $\size(\varphi)=\hsi(\varphi)$.

Also, since $\varphi\in \cO^G(\overline{X})$, it satisfies the control condition
at $\infty$, hence, there exists $E\in \cE_{Gcc}^X$ such that 
\[\supp(\varphi)\subseteq \left\{ (\bar{x},t,\bar{y},s) \in (\overline{X}\times[1,\infty))^2 : (p(\bar{x}),t,p(\bar{y}),s) \in E\right\}.\]
By the second condition in the definition of $\cE_{Gcc}^X$, there exists $\delta >0$ such that
\begin{equation}\label{vsz}
\forall (\bar{x},t,\bar{y},s)\in \supp(\varphi), |t-s|\leq \delta.
\end{equation}
We can then define the \emph{vertical size} of $\varphi$ by
\[\vsi(\varphi)=\inf\left\{\delta\geq 0 : \delta \text{ satisfies } \eqref{vsz}\right\}.\]

\end{defn}

\subsection {Assembly map.} In \cite{dl}, Davis and L\"uck associate to every $G$-$CW$-complex a spectrum $H^G(X,\mathbf{K}(R))$ whose homotopy groups define a $G$-equivariant homology theory with the following property:
\[
H^G_*(G/H, \textbf{K}(R)) \cong K_*(RH), \qquad \forall H \text{ subgroup of } G.
\]

Let $\mathcal{F}$ be a family of subgroups of $G$, i.e. a nonempty collection of subgroups closed under conjugation and subgroups. The classifying space $E_{\mathcal{F}}G$ is the universal $G$-space for actions with isotropy in $\cF$. This is a $G$-$CW$-complex characterized up to $G$-homotopy equivalence by the property that, for any subgroup $H$ of $G$, the $H$-fixed point space $E_{\mathcal{F}}G^H$ is empty if $H\notin \mathcal{F}$, and contractible if $H\in \mathcal{F}$. Note that when $\cF$ is just the trivial subgroup, the space $E_\cF G$ is the usual classifying space $EG$. 

The \emph{assembly map} is the map induced by the projection to the one point space $E_{\mathcal{F}}G\to G/G=\pt$:
\begin{equation}\label{assem}\text{assem}_{\mathcal{F}}: H^G(E_{\mathcal{F}}G,\textbf{K}(R)) \to H^G(G/G,\textbf{K}(R)) \cong \textbf{K}(RG).
\end{equation}

For $\cF=\mathcal{V}cyc$ the family of virtually cyclic subgroups, the \emph{Farrell-Jones conjecture} asserts that the assembly map 
\begin{equation}\label{f-j}\text{assem}_{\mathcal{V}cyc_*}: H^G_*(E_{\mathcal{V}cyc}G,\textbf{K}(R)) \to H^G_*(\pt,\textbf{K}(R)) \cong K_*(RG)
\end{equation}
is an isomorphism.

The assembly map can also be interpreted through means of controlled topology, as we proceed to explain. Using the map induced by the inclusion $\left\{1\right\}\subset [1,\infty)$ and the quotient map we obtain the \emph{germs at infinity sequence}
\[ \cT^G(\overline{X})\to\cO^G(\overline{X})\to \cD^G(\overline{X}),\]
where the inclusion can be identified with a Karoubi filtration and $\cD^G(\overline{X})$ with its quotient (see \cite{blr}*{Lemma 3.6}). Hence, there is a homotopy fibration sequence in $K$-theory:
\begin{equation}\label{seq}\textbf{K}(\cT^G(\overline{X}))\to\textbf{K}(\cO^G(\overline{X}))\to\textbf{K}(\cD^G(\overline{X})).\end{equation}

The functor $X\to \textbf{K}(\cD^G(X))$ is a $G$-equivariant homology theory on $G$-$CW$-complexes, and its value at $G/H$ is weakly equivalent to $\Sigma\textbf{K}(RH)$ \cite{bfjr}*{Sections 5 and 6}.

Applying \eqref{seq}  to the projection $E_{\cF}G\to \pt$ we obtain the following commutative diagram with exact rows:
\begin{equation}\label{ladder}
\xymatrix@C13pt{{...}\ar[r]&{K_n(\cO^G(E_{\cF}G))}\ar[r]\ar[d]&{K_n(\cD^G(E_{\cF}G))}\ar[r]^-{\beta_n}\ar[d]_{\alpha_n}&{K_{n-1}(\cT(E_{\cF}G))}\ar[d]^{\gamma_{n-1}}\ar[r]&{...} \\ {...}\ar[r]&{K_n(\cO^G(\pt))}\ar[r]&{K_n(\cD^G(\pt))}\ar[r]_-{\delta_n} &{K_{n-1}(\cT(\pt))}\ar[r]&{...}}
\end{equation}

By \cite{dl}*{Corollary 6.3}, $\alpha_n$ identifies with the assembly map $\text{assem}_{\mathcal{F}_{n-1}}$  \eqref{assem}.
Using the shift $x\mapsto x+1$ it is easy to see that $\cO^G(\pt)$ admits an Eilenberg swindle, hence $K_n(\cO^G(\pt))=0$ and
$\delta_n$ is an isomorphism. Moreover, $\gamma_{n-1}$ is also an isomorphism, because its 
source and target are both isomorphic to $K_{n-1}(RG)$ by Lemma \ref{lem equiv}. 
This explains the choice of notation: $\cO^G(E_\cF G)$ is the \emph{obstruction category}, i.e. the assembly map is a weak equivalence if and only if $K_n(\cO^G(E_{\cF}G))=0$, $\forall n\in\Z$. 

\section{Two examples}\label{sec:twoex}

From now on, we will focus on two particular assembly maps:
\begin{enumerate}[(i)]
    \item\label{BHS} for the group $\ltr$ and the trivial subgroup  family, and
    \item\label{item:Dinf} for the infinite dihedral group $\Dinf$ and the family $\fin$ of finite subgroups.
\end{enumerate}
As we explain below, both assembly maps are isomorphisms if we take a regular ring $R$ as our coefficient ring. For \eqref{BHS}, this amounts to the well-known theorem of Bass-Heller-Swan. For \eqref{item:Dinf}, we will show that the assembly map is an isomorphism using the equivariant Atiyah-Hirzebruch spectral sequence and a computation of $K_q(R\Dinf)$ made by Davis-Khan-Ranicki \cite{ranicki}. Throughout this section, we use no techniques from controlled algebra.

\subsection{The assembly map for \texorpdfstring{$\ltr$}{<t>}}\label{sec:assZ}
As noted in the introduction, a model for $E\ltr$ is the free $\ltr$-space $\R$. The assembly map for $\ltr$ and the trivial subgroup family gives us a morphism
\begin{equation}\label{ass:Z}
\assem_{\cF}:H^\ltr_1(\R, \bfK(R))\to H^\ltr_1(\pt,\bfK(R))\cong K_1(R[t, t^{-1}]).
\end{equation}
We will describe the source of \eqref{ass:Z} in terms of the $K$-theory of $R$ using a Mayer-Vietoris sequence.

We will drop $\mathbf{K}(R)$ from the notation for clarity, and write $H^\ltr_*(?)$ instead of $H^\ltr_*(?,\bfK(R))$.

When regarding $\R$ as a $\ltr$-CW-complex, we only have one $\ltr$-0-cell, which is compromised by the integers $\Z \subseteq \R$.
Then we have one $\ltr$-1-cell with 
attaching map $\alpha:\ltr \times \{0,1\} \to \Z$ defined as 
$\alpha(t^k,\epsilon)= k+\epsilon$.
The following adjunction space pushout describes $\R$ as a $\ltr$-CW-complex:
\[\begin{tikzpicture}
      \matrix (m) [matrix of math nodes, row sep=3em,
      column sep=3em, text height=1.5ex, text depth=0.25ex]
      { \langle t \rangle \times \{0,1\} & \Z \\
        \langle t \rangle \times [0,1] & \R \\};
      \path[->]
      (m-1-1) edge node[auto] {$\alpha$}        (m-1-2)
      (m-1-1) edge node[auto] {$i$}        (m-2-1)
      (m-2-1) edge node[auto] {$\tilde{\alpha}$}        (m-2-2)
      (m-1-2) edge node[auto] {$j$}       (m-2-2);
  \end{tikzpicture}\]
This pushout gives a long Mayer-Vietoris sequence in homology as follows:
\[
  \begin{tikzpicture}
      \matrix (m) [matrix of math nodes, row sep=2.25em,
      column sep=2em, text height=1.4ex, text depth=0.5ex]
      { \ldots &\Ht_1(\ltr \times \{0,1\}) & \Ht_1(\Z) \oplus \Ht_1(\ltr \times [0,1]) \\
               &                           & \Ht_1(\R)                               \\
       \ldots  & \Ht_0(\Z) \oplus \Ht_0(\ltr \times [0,1]) 
                                           & \Ht_0(\ltr \times \{0,1\})             \\};
      \begin{scope}[every node/.style={scale=.75}]                                     
      \path[->]
      (m-1-1) edge node[auto] {}        (m-1-2)
      (m-1-2) edge node[auto] {$(\alpha_*,i_*)$}        (m-1-3)
      (m-1-3) edge node[auto] {$(j_*,-\tilde{\alpha}_*)$}        (m-2-3.north)
      (m-2-3) edge node[auto] {$\partial$}        (m-3-3)
      (m-3-3) edge node[auto] {$(\alpha_*,i_*)$}        (m-3-2)
      (m-3-2) edge node[auto] {}        (m-3-1);
      \end{scope}
      \end{tikzpicture}
      \]
We are only interested in calculating $H^\ltr_1(\R)$
so we will only use the five terms of the sequence depicted here. 
Since $\Ht_*$ is an equivariant homology theory and
$\ltr\times[0,1]$ is equivariantly homotopy equivalent to $\ltr$, the sequence before
is isomorphic to the following one:
\[
  \begin{tikzpicture}
      \matrix (m) [matrix of math nodes, row sep=2.25em,
      column sep=3em, text height=1ex, text depth=0.25ex]
      { \Ht_1(\ltr)^2    & \Ht_1(\ltr)^2 & \Ht_1(\R) \\
                         & \Ht_0(\ltr)^2 & H_0(\ltr)^2 & \\};
      \begin{scope}[every node/.style={scale=.75}]
      \path[->]
      (m-1-1) edge node[auto] {$(\alpha_*,i_*)$}        (m-1-2)
      (m-1-2) edge node[auto] {$(j_*,-\tilde{\alpha}_*)$}        (m-1-3)
      (m-1-3) edge node[auto] {$\partial$}        (m-2-3)
      (m-2-3) edge node[auto] {$(\alpha_*,i_*)$}        (m-2-2);
      \end{scope}
      \end{tikzpicture}
      \]
On the other hand, since we are taking coefficients in the spectrum $\mathbf{K}(R)$,
we have that $\Ht_*(\langle t \rangle) = K_*(R)$.
We observe now that $\alpha_*$ can be
expressed as $(id,id)$: the morphism $\alpha_*$ is defined as the identity on the
first summand and as the shift on the second. The shift induces the identity because
on the $K$-theory of the $R$-linear category $R\mathcal{G}^\ltr(\ltr)$ --- with 
$\mathcal{G}^\ltr(\ltr)$ being the transport groupoid of $\ltr$ ---
the inclusion of $\ltr$ in any element of $\mathcal{G}^\ltr(\ltr)$
induces an equivalence (see \cite{dl}*{Section 2}). Since it induces the same 
equivalence in every inclusion, it is easily seen that it must be the identity.
This implies that the cokernel of $(\alpha_*,i_*):K_1(R)^2 \to K_1(R)^2 $
is isomorphic to $K_1(R)$. Similarly, the kernel of 
$(\alpha_*,i_*):K_0(R)^2 \to K_0(R)^2$ is isomorphic to $K_0(R)$. This gives the following short
exact sequence:
\begin{equation}\label{suc-z}
     0 \to K_1(R) \xrightarrow{\phi} \Ht_1(\R) \to K_0(R) \to 0
\end{equation}
We will now show that this sequence splits. Note as $q:\Ht_1(\ltr)^2 \to \Ht_1(\ltr)$
the quotient map to the cokernel of $(\alpha_*,i_*)$.
The maps induced by the projections $\R \to \pt$ and $\ltr \to \pt$
give a commutative diagram:
\[
  \begin{tikzpicture}
      \matrix (m) [matrix of math nodes, row sep=3em,
      column sep=4em, text height=1.5ex, text depth=0.25ex]
      { \Ht_1(\ltr)^2 & \Ht_1(\R) \\
        \Ht_1(\ltr)   & \Ht_1(\pt) \\};
        \begin{scope}[every node/.style={scale=.75}]
      \path[->]
      (m-1-1) edge node[auto] {$(j_*,-\tilde{\alpha}_*)$}        (m-1-2)
      (m-1-2) edge node[auto] {}        (m-2-2)
      (m-1-1) edge node[left] {$q$}     (m-2-1)
      (m-2-1) edge node[auto] {$\phi$}     (m-1-2)
      (m-2-1) edge node[auto] {}        (m-2-2);
      \end{scope}
      \end{tikzpicture}
      \]
Since $\Ht_1(\pt)=K_1(R\ltr)$ and $\Ht_1(\ltr)=K_1(R)$, 
we can define a map $r:\Ht_1(\pt)\to \Ht_1(\ltr)$, which is
induced in $K$-theory by the ring morphism $R\ltr \to R$ sending
$t$ to $1$. Due to the commutativity of the previous diagram, 
the map $r$ composed with the induced map of the projection $\R \to \pt$
splits \eqref{suc-z}, meaning that $\Ht_1(\R)\cong K_1(R)\oplus K_0(R)$.

\subsection{The assembly map for \texorpdfstring{$\Dinf$}{Dinf}} 
We use the presentation \eqref{presDinf} for $\Dinf$. It is easily seen that every element of $\Dinf$ can be uniquely written as $r^ms^n$ with $m\in\Z$ and $n\in\left\{0,1\right\}$. Moreover, any non-trivial subgroup of $\Dinf$ belongs to one of the following families:
\begin{enumerate}
    \item\label{item:sim} $\langle r^ms \rangle$ with $m\in\Z$,
    \item\label{item:infcyc} $\langle r^m \rangle$ with $m\in\N$,
    \item\label{item:infdih} $\langle r^m, r^ks \rangle$ with $m\in\N$ and $k\in \Z$.
\end{enumerate}
Subgroups of type \eqref{item:sim} have order $2$, those of type \eqref{item:infcyc} are infinite cyclic and those of type \eqref{item:infdih} are isomorphic to $\Dinf$. We will write $H_m$ for the subgroup $\langle r^m s\rangle$ --- note that these are different for different values of $m\in\Z$. Let $\fin$ be the family of finite subgroups of $\Dinf$, consisting of the trivial subgroup
and those subgroups of type \eqref{item:sim}. We are interested in the assembly map for $\Dinf$ and the family $\fin$.

Let $\Dinf$ act on $\R$ on the left by putting
\[r^ms^n\cdot x=m+(-1)^nx\]
for $m\in \Z$, $n\in\left\{0,1\right\}$ and $x\in\R$. The element $r^m$ acts by translation by $m$ and $r^ms$ acts by symmetry with respect to the point $\frac{m}{2}$. Then $E_\fin\Dinf=\R$ since $\R^H=\emptyset$ for $H\not\in\fin$ and $\R^H$ is contractible for $H\in\fin$.
In what follows, we will show that the assembly map for $\Dinf$ and the family $\fin$,
\begin{equation}\label{ass:Dinf}\assem_{\fin}:H^\Dinf_q(\R, \bfK(R))\to H^\Dinf_q(\pt,\bfK(R))\cong K_q(R\Dinf),\end{equation}
is an isomorphism for regular $R$.

For a $\Dinf$-CW-complex $X$, the equivariant Atiyah-Hirzebruch spectral sequence \cite{libroluck}*{Example 10.2} converges to $H_*^\Dinf(X, \bfK(R))$. Its second page is given by:
\[E^2_{pq}=H_p\left(C_\bullet^\OrDinf(X)\otimes_\OrDinf H^\Dinf_q(?,\bfK(R))\right)\]
Here, $C_\bullet^\OrDinf(X):(\OrDinf)^\op\to \Ch$ is the functor that sends a coset $\Dinf/H$ to the cellular chain complex of $X^H$, $H^\Dinf_q(?,\bfK(R)):\OrDinf\to\Ab$ is the restriction of the equivariant homology with coefficients in $\bfK(R)$, and $\otimes_\OrDinf$ stands for the balanced tensor product.
Let us compute the left hand side of \eqref{ass:Dinf}. If $H\subset\Dinf$ is a subgroup of types \eqref{item:infcyc} or \eqref{item:infdih}, then $\R^H=\emptyset$ and so $C_\bullet^\OrDinf(\R)(\Dinf/H)$ is the zero chain complex. Since $\R^{H_m}=\{\frac{m}{2}\}$, $C_\bullet^\OrDinf(\R)(\Dinf/H_m)$ is the abelian group $\Z$ --- generated by the $0$-cell $\frac{m}{2}$ --- concentrated in degree $0$. Finally, $C_\bullet^\OrDinf(\R)(\Dinf/1)$ is the complex
\[\xymatrix{\cdots\ar[r] & 0\ar[r] & \Z^{(\Dinf)}\ar[r]^-{d_1} & \Z^{(\Z)}\oplus\Z^{\left(\Z+\frac12\right)}\ar[r] & 0}\]
where $d_1$ acts on the basic element $g\in\Dinf$ by $d_1(g)=g\cdot \frac12-g\cdot 0$. To ease notation, we will write $A_\bullet$ instead of $C_\bullet^\OrDinf(\R)\otimes_\OrDinf H^\Dinf_q(?,\bfK(R))$. Taking the above into account, $A_n=0$ for $n\neq 0,1$. Moreover, we have
\begin{equation}\label{eq:c1}A_1=\frac{\Z^{(\Dinf)}\otimes K_q(R)}{N_1}\end{equation}
\begin{equation}\label{eq:c0}A_0=\frac{\left(\Z^{(\Z)}\oplus\Z^{(\Z+\frac12)}\right)\otimes K_q(R)\oplus \bigoplus_{m\in\Z}\Z\otimes K_q(RH_m)}{N_0}\end{equation}
where $N_i$ is the subgroup generated by the elements $f^*(x)\otimes y-x\otimes f_*(y)$ for all morphisms $f:\Dinf/H\to \Dinf/K$ in $\OrDinf$. For $A_1$ we only have to consider morphisms $f:\Dinf/1\to\Dinf/1$ and these induce the identity upon applying $H^{\Dinf}_*(?,\bfK(R))$. It follows that $N_1$ is generated by the elements of the form $f^*(x)\otimes y-x\otimes y$. Since $\Dinf$ acts transitively on the $1$-cells of $\R$, all the copies of $K_q(R)$ in the right hand side of \eqref{eq:c1} become identified after dividing by $N_1$ and hence \eqref{eq:c1} is isomorphic to $K_q(R)$. To be precise, any of the inclusions $K_q(R)\to \Z^{(\Dinf)}\otimes K_q(R)$ corresponding to an element of $\Dinf$ induces the same isomorphism $K_q(R)\cong A_1$. The situation for $A_0$ is slightly more complicated: we have to consider morphisms $\Dinf/1\to\Dinf/1$, $\Dinf/1\to \Dinf/H_m$ and $\Dinf/H_m\to\Dinf/H_n$. It is easily verified that
\[\Hom_{\OrDinf}\left(\Dinf/H_m, \Dinf/H_n\right)=\left\{\begin{array}{cc}
    \{\ast\} & \text{if $m\equiv n$ (mod $2$),} \\
    \emptyset & \text{otherwise.}
\end{array}\right.\]
It follows that each summand in $\bigoplus_{m\in\Z}\Z\otimes K_q(RH_m)$ is identified either with $K_q(RH_0)$ or with $K_q(RH_1)$ upon dividing by $N_0$. The action of $\Dinf$ on the $0$-cells of $\R$ has two orbits: $\Z$ and $\Z+\frac12$. It follows that $\left(\Z^{(\Z)}\oplus\Z^{(\Z+\frac12)}\right)\otimes K_q(R)$ is identified with $K_q(R)\oplus K_q(R)$ after dividing by $N_0$ --- one copy of $K_q(R)$ for each orbit of the $0$-cells. Finally, any morphism $\Dinf/1\to\Dinf/H_m$ induces the natural morphism $K_q(R)\to K_q(RH_m)$ upon applying $H^\Dinf_q(?,\bfK(R))$. Thus, the copy of $K_q(R)$ corresponding to the orbit $\Z$ is identified with its image in $K_q(RH_0)$, and the copy of $K_q(R)$ corresponding to $\Z+\frac12$ is identified with its image in $K_q(RH_1)$. Hence \eqref{eq:c0} is isomorphic to $K_q(RH_0)\oplus K_q(RH_1)$. To be precise, the inclusion
\[\xymatrix{K_q(RH_0)\oplus K_q(RH_1)\ar[r] & \bigoplus_{m\in\Z}\Z\otimes K_q(RH_m)}\]
induces an isomorphism $K_q(RH_0)\oplus K_q(RH_1)\cong A_0$. Write $\iota_m:K_q(R)\to K_q(RH_m)$ for the split monomorphism induced by the inclusion $1\subset H_m$. It can be shown that $A_\bullet$ is the complex:
\[\xymatrix{\cdots \ar[r] & 0\ar[r] & K_q(R)\ar[r]^-{(-\iota_0, \iota_1)} & K_q(RH_0)\oplus K_q(RH_1)\ar[r] & 0}\]
Upon taking homology, we get the second page of the Atiyah-Hirzebruch spectral sequence converging to $H^\Dinf_*(\R, \bfK(R))$:
\[E^2_{pq}=H_p(A_\bullet)= \left\{\begin{array}{cc}
K_q(RH_0)\oplus_{K_q(R)}K_q(RH_1) & \text{if $p=0$,} \\
0 & \text{otherwise.}
\end{array}\right.\]
Note that this is actually the infinity-page, and it is easily deduced from it that $H^\Dinf_q(\R,\bfK(R))\cong K_q(RH_0)\oplus_{K_q(R)}K_q(RH_1)$.

Let $B_\bullet$ be the chain complex $C^\OrDinf_\bullet(\pt)\otimes_\OrDinf H^\Dinf_q(?,\bfK(R))$. Note that, for any subgroup $H\subseteq \Dinf$, $C^\OrDinf_\bullet(\pt)(\Dinf/H)$ is the abelian group $\Z$ concentrated in degree $0$. Then $B_n=0$ for $n\neq 0$ and we have:
\[B_0=\frac{\bigoplus_HK_q(RH)}{N_0}\]
Obviously, $B_0\cong K_q(R\Dinf)$. But it follows from \cite{ranicki}*{Corollary 3.27} that, when $R$ is regular, the inclusion $K_q(RH_0)\oplus K_q(RH_1)\to\bigoplus_H K_q(RH)$ induces an isomorphism:
\[\xymatrix{K_q(RH_0)\oplus_{K_q(R)} K_q(RH_1)\ar[r]^-{\cong} & B_0}\]
It is easily seen that the projection $\R\to \pt$ induces the following morphism of chain complexes:
\[\xymatrix@C=2em{A_\bullet\ar[d] & \cdots \ar[r] & 0\ar[r]\ar[d] & K_q(R)\ar[d]\ar[r]^-{(-\iota_0, \iota_1)} & K_q(RH_0)\oplus K_q(RH_1)\ar@{->>}[d]\ar[r] & 0\ar[d] \\
B_\bullet & \cdots \ar[r] & 0\ar[r] & 0\ar[r] & K_q(RH_0)\oplus_{K_q(R)} K_q(RH_1)\ar[r] & 0}\]
Since this is a quasi-isomorphism, the morphism induced between the second pages of the Atiyah-Hirzebruch spectral sequence is an isomorphism. As we have already seen, we can recover $H^\Dinf_q(\R,\bfK(R))$ from this spectral sequence. Hence, the above shows that the assembly map \eqref{ass:Dinf} is an isomorphism.

\begin{rem}
The fact that \eqref{ass:Dinf} is an isomorphism for regular $R$ also follows from \cite{lr}*{Theorem 65} and \cite{dqr}*{Theorem 2.1}.
\end{rem}

\section{Vanishing theorem for \texorpdfstring{$K_1$}{K1}}\label{sec:vanishing}
In this section we show that
automorphisms in $\cO^{\ltr}(\R)$ and $\cO^{\Dinf}(\R)$ with small enough horizontal size have
trivial class in $K_1$.

In the case of $\ltr$ with the trivial subgroup family, the action of $\ltr$ on $E\ltr = \R$ is free.
Then, using Corollary \ref{tilde}, we can use the category $\OGR$ instead of $\cO^{\ltr}(\R)$.
On the other hand, in the case of $\Dinf$ with the family $\cF in$, the action of $\Dinf$ on $E_{\fin}\Dinf=\R$
is not free. But instead of using the standard resolution, we use Lemma \ref{restriccion-libre} 
observing that $\langle r \rangle$ acts freely on $\R$ and $\Dinf/\langle r \rangle
\cong \Z/2$. Hence we can use the resolution $p:\overline{\R}:=\Z/2 \times \R \to \R$
given by the projection and the diagonal action on $\overline{\R}$.
As in Corollary \ref{tilde} we can use the category $\ODR$ instead
of $\cO^{\Dinf}(\R)$ with the standard resolution.
We interpret the resolution $\overline{\R}$ as two different copies of $\R$.

\begin{rem}\label{d-en-z}
The category $\ODR$ can be embedded into
$(\OGR)^{\oplus2}$ by restricting the modules over each copy of $\R\times[1,\infty)$. Because
the action of $r \in \Dinf$ is the same as the action of $t$ in $\R$,
the restricted modules are
also $\ltr$-equivariant. The control conditions are trivially satisfied. It is also important
to note that the control condition on the morphisms is given by Remark \ref{times-control}, this
means that the morphisms on this category can have non-null coordinates from one copy to another
at arbitrary height, that is, this control condition does not control the distance between the copies of 
$\overline{\R}$.
\end{rem}
\begin{rem}\label{rem:s-chico}
Set $\varphi:M\to M$ in $\cT^\Dinf(\overline{\R})$ as
\begin{align*}
 M_{(\epsilon,x)} &= \begin{cases*}
                    R   &if $x\in \Z$, \\
                    0   &otherwise,
                    \end{cases*} \\
\varphi_{(\epsilon,x)}^{(\delta,y)} &= \begin{cases*}
                                    id  &if $\epsilon \neq \delta$ and $x=y \in \Z$, \\
                                    0   &otherwise.
                                    \end{cases*}
\end{align*}
Then $\varphi$ is clearly an automorphism and it is easily seen that 
 $[\mathbb{U}(\varphi)] = [s]\in K_1(R\Dinf)$. Because of the earlier remark,
 $\varphi$ is a $0$-automorphism, so the class $[s]$ is automatically small.
\end{rem}

\subsection{The swindle}
Let $\alpha: A \to B$ be an morphism in $\OGR$ and $I$ an open interval in $\R$ of length 1.
We say that $\alpha$ \emph{restricts to $I$} if $\supp(A)$ and $\supp(B)$ do not intersect
$\partial I \times [1,+\infty)$ and for every $(x,t) \in I \times [1,+\infty)$ and every
$(y,s) \notin I \times [1,+\infty)$ then $\alpha_{(x,t)}^{(y,s)} = 0$ and $\alpha_{(y,s)}^{(x,t)} = 0$.

Let $\alpha:A \to B$ be a morphism in $\OGR$ that restricts to an interval $I$.
Set then $I = (a,a+1)$.
We proceed to squeeze $\alpha$ to $I$ in the following way. 

For each natural number let $f_n:\R \to \R$ be the function that at each interval
$I+k$ interpolates
linearly the endpoints to $a+k+\frac{1}{2}-\frac{1}{2n}$ and
$a+k+\frac{1}{2}+\frac{1}{2n}$ respectively,
\[ f_n(x) = a+k+\frac{1}{2} - \frac{1}{2n} + \frac{(x-a-k)}{n} \text{ if } x\in[a+k,a+k+1).\]
Note that since each $f_n$ is defined over an interval of length 1,
we have that $f_n(x+1)=f_n(x)+1$
and thus, is equivariant for the action of $\langle t \rangle$.
Although each $f_n$ depends on the choice of the interval $I$, we drop it from the notation
as it will be clear from the context which interval we are using. Observe that $f_1=id$.
    
Let $(\tau_n)_{n\in\N}$ be an unbounded strictly increasing sequence of 
real numbers in $[1,\infty)$ with $\tau_1=1$.
We define the \emph{$n$th layer of the squeezing of $A$ (with
heights $\tau_n$)}, noted by $S_n(A)$, as the geometric module given by
\begin{equation}
    S_n(A)_{(x,t)} = \begin{cases*}
                    A_{(f_{n}^{-1}(x),t+1-\tau_n)} &if  $t\geq \tau_n$
                                                    and $x\in \mathrm{Im} f_n$ \\
                    0                              &if else.
                    \end{cases*}
\end{equation} 
 Note that $S_1(A)=A$.

Intuitively, $S_n(A)$ is the geometric module given by rising $A$
to $\tau_{n}$ and then squeezing it to the midpoint of the interval $I$.

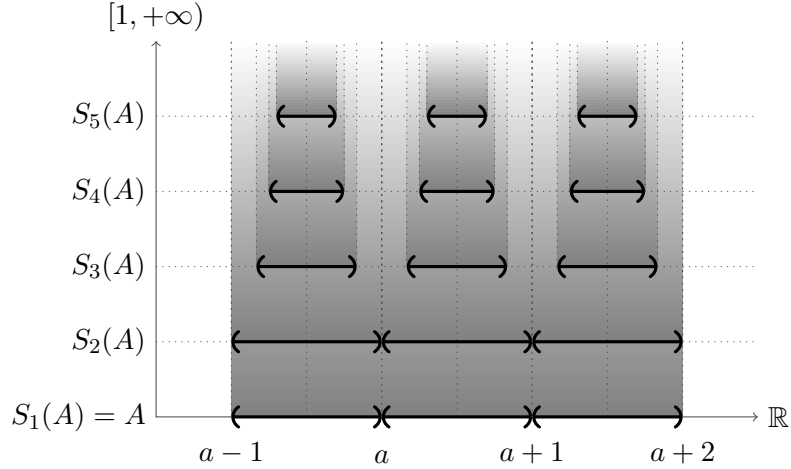
\[\begin{tikzpicture}
\draw[<->, black!70] (8,0) node[right, black]{$\mathbb{R}$} -- (0,0) --
(0,5) node[above, black]{$[1,+\infty)$};

\shade[bottom color=gray, top color=white] (2-1,0) -- (2-1,5) -- (2+1,5) -- (2+1,0) -- cycle;
\shade[bottom color=gray, top color=white] (2-1,1) -- (2-1,5) -- (2+1,5) -- (2+1,1) -- cycle;
\shade[bottom color=gray, top color=white] (2-2/3,2) -- (2-2/3,5) -- (2+2/3,5) -- (2+2/3,2) -- cycle;
\shade[bottom color=gray, top color=white] (2-2/4,3) -- (2-2/4,5) -- (2+2/4,5) -- (2+2/4,3) -- cycle;
\shade[bottom color=gray, top color=white] (2-2/5,4) -- (2-2/5,5) -- (2+2/5,5) -- (2+2/5,4) -- cycle;

\shade[bottom color=gray, top color=white] (4-1,0)   -- (4-1,5)   -- (4+1,5)   -- (4+1,0) -- cycle;
\shade[bottom color=gray, top color=white] (4-1,1)   -- (4-1,5)   -- (4+1,5)   -- (4+1,1) -- cycle;
\shade[bottom color=gray, top color=white] (4-2/3,2) -- (4-2/3,5) -- (4+2/3,5) -- (4+2/3,2) -- cycle;
\shade[bottom color=gray, top color=white] (4-2/4,3) -- (4-2/4,5) -- (4+2/4,5) -- (4+2/4,3) -- cycle;
\shade[bottom color=gray, top color=white] (4-2/5,4) -- (4-2/5,5) -- (4+2/5,5) -- (4+2/5,4) -- cycle;

\shade[bottom color=gray, top color=white] (6-1,0)   -- (6-1,5)   -- (6+1,5)   -- (6+1,0) -- cycle;
\shade[bottom color=gray, top color=white] (6-1,1)   -- (6-1,5)   -- (6+1,5)   -- (6+1,1) -- cycle;
\shade[bottom color=gray, top color=white] (6-2/3,2) -- (6-2/3,5) -- (6+2/3,5) -- (6+2/3,2) -- cycle;
\shade[bottom color=gray, top color=white] (6-2/4,3) -- (6-2/4,5) -- (6+2/4,5) -- (6+2/4,3) -- cycle;
\shade[bottom color=gray, top color=white] (6-2/5,4) -- (6-2/5,5) -- (6+2/5,5) -- (6+2/5,4) -- cycle;

\draw[dotted, black!70] (0,1) node[left, black]{$S_2(A)$} -- (8,1);
\draw[dotted, black!70] (0,2) node[left, black]{$S_3(A)$} -- (8,2);
\draw[dotted, black!70] (0,3) node[left, black]{$S_4(A)$} -- (8,3);
\draw[dotted, black!70] (0,4) node[left, black]{$S_5(A)$} -- (8,4);
\draw (0,0) node[left]{$S_1(A)=A$};

\draw[dotted, black!70] (2,5) -- (2,0);
\draw[dotted, black!70] (4,5) -- (4,0);
\draw[dotted, black!70] (6,5) -- (6,0);

\draw (1,-0.75) node[above]{$a-1$};
\draw (3,-0.75) node[above]{$a$};
\draw (5,-0.75) node[above]{$a+1$};
\draw (7,-0.75) node[above]{$a+2$};

\newlength{\gr}
\setlength{\gr}{1.2pt}

\draw[(-), line width=\gr] (2-1,0) -- (2+1,0);
\draw[(-), line width=\gr] (2-1,1) -- (2+1,1);
\draw[(-), line width=\gr] (2-2/3,2) -- (2+2/3,2);
\draw[(-), line width=\gr] (2-2/4,3) -- (2+2/4,3);
\draw[(-), line width=\gr] (2-2/5,4) -- (2+2/5,4);

\draw[dotted, black!70] (2-1,0)   -- (2-1,5);
\draw[dotted, black!70] (2-1,1)   -- (2-1,5);
\draw[dotted, black!70] (2-2/3,2) -- (2-2/3,5);
\draw[dotted, black!70] (2-2/4,3) -- (2-2/4,5);
\draw[dotted, black!70] (2-2/5,4) -- (2-2/5,5);
\draw[dotted, black!70] (2+1,0)   -- (2+1,5);
\draw[dotted, black!70] (2+1,1)   -- (2+1,5);
\draw[dotted, black!70] (2+2/3,2) -- (2+2/3,5);
\draw[dotted, black!70] (2+2/4,3) -- (2+2/4,5);
\draw[dotted, black!70] (2+2/5,4) -- (2+2/5,5);

\draw[(-), line width=\gr] (4-1,0) -- (4+1,0);
\draw[(-), line width=\gr] (4-1,1) -- (4+1,1);
\draw[(-), line width=\gr] (4-2/3,2) -- (4+2/3,2);
\draw[(-), line width=\gr] (4-2/4,3) -- (4+2/4,3);
\draw[(-), line width=\gr] (4-2/5,4) -- (4+2/5,4);

\draw[dotted, black!70] (4-1,0)   -- (4-1,5);
\draw[dotted, black!70] (4-1,1)   -- (4-1,5);
\draw[dotted, black!70] (4-2/3,2) -- (4-2/3,5);
\draw[dotted, black!70] (4-2/4,3) -- (4-2/4,5);
\draw[dotted, black!70] (4-2/5,4) -- (4-2/5,5);
\draw[dotted, black!70] (4+1,0)   -- (4+1,5);
\draw[dotted, black!70] (4+1,1)   -- (4+1,5);
\draw[dotted, black!70] (4+2/3,2) -- (4+2/3,5);
\draw[dotted, black!70] (4+2/4,3) -- (4+2/4,5);
\draw[dotted, black!70] (4+2/5,4) -- (4+2/5,5);

\draw[(-), line width=\gr] (6-1,0) -- (6+1,0);
\draw[(-), line width=\gr] (6-1,1) -- (6+1,1);
\draw[(-), line width=\gr] (6-2/3,2) -- (6+2/3,2);
\draw[(-), line width=\gr] (6-2/4,3) -- (6+2/4,3);
\draw[(-), line width=\gr] (6-2/5,4) -- (6+2/5,4);

\draw[dotted, black!70] (6-1,0)   -- (6-1,5);
\draw[dotted, black!70] (6-1,1)   -- (6-1,5);
\draw[dotted, black!70] (6-2/3,2) -- (6-2/3,5);
\draw[dotted, black!70] (6-2/4,3) -- (6-2/4,5);
\draw[dotted, black!70] (6-2/5,4) -- (6-2/5,5);
\draw[dotted, black!70] (6+1,0)   -- (6+1,5);
\draw[dotted, black!70] (6+1,1)   -- (6+1,5);
\draw[dotted, black!70] (6+2/3,2) -- (6+2/3,5);
\draw[dotted, black!70] (6+2/4,3) -- (6+2/4,5);
\draw[dotted, black!70] (6+2/5,4) -- (6+2/5,5);
\end{tikzpicture}\]

Observe that $A$ and $S_n(A)$ are isomorphic through the isomorphism that sends $A_{(x,t)}$ to
$S_n(A)_{(f_n(x),t+\tau_n-1)} = A_{(x,t)}$ with the identity.
The corresponding morphism to $\alpha$ through the isomorphism just described is noted as
$S_n(\alpha):S_n(A)\to S_n(B)$. 

Define $S(A) = \bigoplus_{n\in \N} S_n(A)$ given by the sum of all layers.
Since the layers get raised, $S(A)$ has only a finite sum at each height.
Note also that since each layer gets squeezed further and further, there is a well
defined morphism $S(\alpha):S(A) \to S(B)$ given by $S_n(\alpha)$ at the layer $n$.
The squeezing, and the fact that $\alpha$ does not have non null coordinates between
intervals, guarantee that this endomorphism satisfies the control condition at
$\infty$.

\begin{defn}\label{squeeze}
    Given an open interval $I\subseteq \R$, let $\OGR_I$ be the subcategory of $\OGR$
    given by the objects which support does not intersect $\partial I \times [1,+\infty)$
    and morphisms of $\OGR$ that restrict to the interval $I$. 

    If also given an strictly increasing unbounded sequence $(\tau_n)_n$
    with $\tau_1 = 1$, the construction just described defines
    functors $S_n:\OGR_I \to \OGR_I$  and $S: \OGR_I \to \OGR_I$.
\end{defn}

Note that these constructions can also be made on the category $\mathcal{O}(I)$
with no action of $\ltr$. In this case, we only consider the interval $I$,
so the functions $f_n$ are defined within $I$.
In this way we also get functors $S_n:\cO(I)\to \cO(I)$ and $S:\cO(I) \to \cO(I)$.

\begin{rem}
We can also define a category 
$\ODR_I$ of objects that restrict to an interval on each copy ${0}\times \R \times [1,\infty)$
and ${1} \times \R \times [1,\infty)$. It is easy to see, using Remark \ref{d-en-z},
that the functors $S$ and $S_n$ also give corresponding 
functors over $\ODR_I$, as the same
construction applies to each copy of $\R \times [1,\infty)$.
\end{rem}

\begin{prop}
    The categories $\OGR_I$ and $\mathcal{O}(I)$ are isomorphic.
\end{prop}
\begin{proof}
 We have a functor $F:\OGR_I \to \mathcal{O}(I)$ given by restriction to $I$ which sends
 each morphism $\phi:A \to B$ that restricts to $I$ to
 \[
    F(\phi)_{(x,t)}^{(y,s)} = \phi_{(x,t)}^{(y,s)} : A_{(x,t)} \to B_{(y,s)}
 \]
 for each $x,y \in I$ and $t,s \in [1,+\infty)$. This functor has an inverse
 $G:\cO(I) \to \OGR_I$ given by repeating the same morphism
 $\psi:M \to N$ over each translation of $I$; for each $x,y \in I$ and $k \in \Z$ set
 \[ 
    G(M)_{(x+k,t)}= M_{(x,t)} \qquad G(N)_{(y+k,s)}= N_{(y,s)} 
    \]
    \[
    G(\psi)_{(x+k,t)}^{(y+k,s)} = \psi_{(x,t)}^{(y,s)}: M_{(x,t)} \to N_{(y,s)}.
 \] 
 It is easily checked that both compositions of $F$ and $G$ give the 
 corresponding identities.
\end{proof}

\begin{lem}\label{lem:K1vanishing}
Set $\tau_n=n$. Then, the functor $S$ makes the category $\mathcal{O}(I)$ flasque,
i.e. there is a natural isomorphism $ S \oplus \id \to S$. In particular,
$K_*(\mathcal{O}(I)) = K_*(\OGR_I)$ is trivial.
\end{lem}
\begin{proof}
 We define the natural transformation as follows: for each $n\geq 1$ define $\phi_n: S_n(A)
 \to S_{n+1}(A)$ given by matrix coordinates $(\phi_n)_{(x,t)}^{(y,s)} = 0$ if
 $f_{n+1}^{-1}(y) \neq f_n^{-1}(x)$ or $s \neq t+1$ and if $f_{n+1}^{-1}(y)=f_n^{-1}(x)$
 and $s= t+1$ since
 \[ S_{n+1}(A)_{(y,s)} = A_{(f_{n+1}^{-1}(y),(t+1)-(n+1)+1)} 
                = A_{(f_n^{-1}(x),t-n+1)} = S_n(A)_{(x,t)}, \]
 we set $(\phi_n)_{(x,t)}^{(y,s)} = id$. Note that this is simply the
 composition of isomorphisms
 $S_n(A) \to A \to S_{n+1}(A)$. Also, put $\phi_0: A \to S_1(A) = A$
 equal to the identity of $A$.
 Putting together all of these morphisms into
 $\oplus_n \phi_n: S(A) \oplus A \to S(A)$, defines the natural
 isomorphism we need.
\end{proof}

Similar statements can be made in the case of the category 
$\ODR$ using two copies of the interval $I$.

Given a geometric module $A$ in $\OGR$ or $\ODR$,
there are two splittings of $S(A)$. One is given by $S(A)=\So(A) \oplus
\Se(A)$, where $\So(A)$ and $\Se(A)$ are the sum over the odd and even
layers respectively. The second splitting is $S(A) = A \oplus S_{+}(A)$ where
$S_{+}(A)$ are all the layers in $S(A)$ omitting the first one.  We also have a
splitting $S_{+}(A) = \Se(A) \oplus \Sop(A)$ where $\Sop(A)$ is just
$\So(A)$ with the first layer removed. 

The geometric modules $\So(A)$, $\Se(A)$ and $\Sop(A)$ are isomorphic
through the isomorphism which identifies each layer with the next one (the order is
important!). We note these isomorphisms
\begin{align*}
   \psioe&:\So(A) \to \Se(A) \text{ and} \\
   \psieo&:\Se(A) \to \Sop(A).
\end{align*}

Given a morphism $\alpha:A\to B$ in $\OGR$ or $\ODR$ we define
\begin{align*}
\So(\alpha)&:\So(A) \to \So(B), \\
\Se(\alpha)&:\Se(A) \to \Se(B) \text{ and}\\
\Sop(\alpha)&:\Sop(A) \to \Sop(B)
\end{align*}
given by $S_n(\alpha)$ at the corresponding layers.
Observe that since each of the morphisms are defined layer-wise, we have that
\begin{align*}
    \psioe^{-1}\Se(\alpha)\psioe &= \So(\alpha) \text{ and} \\
    \psieo^{-1}\Sop(\alpha)\psieo &= \Se(\alpha).
\end{align*}  

\begin{rem}Let $\eta:A\to A$ be a morphism in $\OGR$. Then
the horizontal size of $\eta$ is given by the formula
\[\hsi(\eta)=\sup \left\{|x-y| : \eta_{(x,t)}^{(y,s)}\ne 0\right\}.\]
In the case $\eta$ is in $\ODR$, the horizontal size is given by
\[\hsi(\eta)=\sup \left\{|x-y| : \eta_{(\epsilon,x,t)}^{(\delta,y,s)}\ne 0\right\}.\]
\end{rem}

\subsection{Vanishing theorem}

\begin{lem}\label{lem:identity}
Fix the interval $I = (0,1)$ and a sequence $\tau_n$ as in
\ref{squeeze} for the construction of the functors $S_n$.
Consider the following subspaces of $(0,1)\times [1, +\infty)$:
\[U=\left[\tfrac13,\tfrac23\right]\times[1,+\infty)\]
\[V=\bigcup_{n\geq 1}\left[\tfrac12-\tfrac{1}{6n},\tfrac12+\tfrac{1}{6n}\right]\times[\tau_n,\tau_{n+1}).\]
If $\gamma:A\to A$ is an endomorphism in $\cO(I)$ which is the identity on $U$ then
$\S_n(\gamma)$ is the identity on $V$ for all $n$.
\end{lem}
\begin{proof}
For each $n\geq 1$ put:
\[\S_n(U):=\left((0,1)\times [1, \tau_n)\right)\cup\left(\left[\tfrac12-\tfrac{1}{6n},\tfrac12+\tfrac{1}{6n}\right]\times [1,+\infty)\right)\]
Then $U=\S_1(U)$ and $V=\cap_n\S_n(U)$. Note that $f_n$ is linear and increasing on $(0,1)$, and that
\[f_n:(0,1)\to \left(\tfrac12-\tfrac{1}{2n},\tfrac12+\tfrac{1}{2n}\right)\]
is bijective. Then, for $x\in\Im f_n$ and $t\geq \tau_n$, $(x,t)\in \S_n(U)$ if and only if $(f_n^{-1}(x), t-\tau_n+1)\in U$. It follows easily from the latter that $\S_n(\gamma)$ is the identity on $\S_n(U)$, and thus on $V\subseteq \S_n(U)$.
\end{proof}

\begin{thm}[{c.f. \cite{ped}*{Theorem 3.6}}]\label{chico2}
Let $\alpha:A \to A$ be an $\frac{1}{30}$-auto\-morphism in $\OGR$.
Then there is an automorphism $\beta:B \to B$ in $\OGR$ with $[\alpha] = [\beta]$ in
$K_1(\OGR)$ such that $\beta$ restricts to $\left(\frac12,\frac32\right)$.
The same is true with $\ODR$ instead of $\OGR$.

\end{thm}

\begin{proof}
We may assume $A$ does not have modules supported at points whose first coordinates are integers or half-integers; if this were not the case, we could replace $A$ by an isomorphic module obtained by slightly shifting $A$. 

Let $\balpha, \balpham:A \to A$ be the morphisms in $\OGR$ defined by
\[ \left(\balphapm\right)_{(x,t)}^{(y,s)} = \begin{cases} \left(\alpha^{\pm 1}\right)_{(x,t)}^{(y,s)} &\text{if } x,y \in [k,k+1] \text{ for some } k\in \Z,\\
0 & \text{otherwise.}
\end{cases}\]
By construction, $\balpha$ and $\balpham$ restrict to the interval $(0,1)$ --- that is, they are morphisms in $\OGR_{(0,1)}$.

We now choose the sequence $(\tau_n)_{n\geq 1}$ for constructing the functors $\S_n$. Fix $K>0$ so that $\left(\alpha^{\pm 1}\right)_{(x,t)}^{(y,s)}=0$ if $\lvert t-s\rvert>K$. By control at $+\infty$, we may choose a strictly increasing sequence $(\tau_n)_{n\geq 1}$ so that the following holds for all $n$:

\emph{Whenever $s,t\geq \tau_n-5K$ and $\lvert x-y\rvert > \frac{1}{30n}$, $\left(\alpha^{\pm 1}\right)_{(x,t)}^{(y,s)}=0=S_j\left(\balphapm\right)_{(x,t)}^{(y,s)}$ for all $1<j<n$.}
\medskip

Denote endomorphisms of $\S(A) = \So(A) \oplus \Se(A)$ as $2\times 2$ matrices and define $\eta: \So(A) \oplus \Se(A) \to \So(A) \oplus \Se(A)$ as the product:
\begin{align*}
\eta = 
&\begin{pmatrix}
    \id & \psioe^{-1}  \\
    0 & \id
\end{pmatrix}
\begin{pmatrix}
    \id & 0 \\
   -\psioe & \id
\end{pmatrix}
\begin{pmatrix}
    \id & \psioe^{-1} \\
    0 & \id
\end{pmatrix} 
\\
&\begin{pmatrix}
    \id & 0 \\
  \psioe\circ\So\left(\balpham\right) & \id
\end{pmatrix}
\begin{pmatrix}
    \id & -\psioe^{-1}\circ\Se\left(\balpha\right) \\
    0 & \id
\end{pmatrix}
\begin{pmatrix}
    \id & 0 \\
    \psioe\circ\So\left(\balpham\right) & \id
\end{pmatrix}
\end{align*}
Define $\mu:\Se(A) \oplus \Sop(A)\to \Se(A) \oplus \Sop(A)$ by: 
\begin{align*}
\mu = 
&\begin{pmatrix}
    \id & \psieo^{-1} \\
    0 & \id
\end{pmatrix}
\begin{pmatrix}
    \id & 0 \\
   -\psieo & \id
\end{pmatrix}
\begin{pmatrix}
    \id & \psieo^{-1} \\
    0 & \id
\end{pmatrix} 
\\
&\begin{pmatrix}
    \id & 0 \\
  \psieo\circ\Se\left(\balpham\right) & \id
\end{pmatrix}
\begin{pmatrix}
    \id & -\psieo^{-1}\circ\Sop\left(\balpha\right) \\
    0 & \id
\end{pmatrix}
\begin{pmatrix}
    \id & 0 \\
    \psieo\circ\Se\left(\balpham\right) & \id
\end{pmatrix}
\end{align*}
Note that $\eta$ and $\mu$ are automorphisms since they are products of elementary matrices. We also have the followin description:
\[\eta=\begin{pmatrix}
    2\So\left(\balpham\right) - \So\left(\balpham\circ\balpha\circ\balpham\right) & -\psioe^{-1}\circ\Se\left(\balpham\circ\balpha\right)+\psioe^{-1} \\
    -\psioe+\psioe\circ\So\left(\balpha\circ\balpham\right) &
    \Se\left(\balpha\right)
\end{pmatrix}
\]
\[\mu=\begin{pmatrix}
    2\Se\left(\balpham\right) - \Se\left(\balpham\circ\balpha\circ\balpham\right) & -\psieo^{-1}\circ\Sop\left(\balpham\circ\balpha\right)+\psieo^{-1} \\
    -\psieo+\psieo\circ\Se\left(\balpha\circ\balpham\right) &
    \Sop\left(\balpha\right)
\end{pmatrix}
\]

Let $\beta:= \eta\circ(\alpha \oplus \mu)$. Since both $\eta$ and $\mu$ are products of elementary matrices, $[\alpha]=[\beta]$ in $K_1(\OGR)$. We will show that $\beta$ restricts to $(\frac{1}{2},\frac{3}{2})$. In order to obtain a more explicit description of $\beta$, we can consider the matrix representations of $\alpha\oplus\mu$ and $\eta$ with respect to the decomposition
\begin{equation}\label{eq:decompn}\S(A)=\S_1(A)\oplus \S_2(A)\oplus \S_3(A)\oplus \cdots\end{equation}
and then multiply these matrices. This tedious but straightforward computation shows that the matrix representation $\beta=(\beta_{ij})_{i,j\geq 1}$ with respect to \eqref{eq:decompn} can be described as follows. Define:
\begin{align*}
    \gamma_j&:=-3\S_j\left(\balpham\circ\balpha\circ\balpham\right)+2\S_j\left(\balpham\right)+\S_j\left(\balpham\circ\balpha\circ\balpham\circ\balpha\circ\balpham\right) \\
    \delta_j&:=2\S_j\left(\balpha\circ\balpham\right)-\S_j\left(\balpha\circ\balpham\circ\balpha\circ\balpham\right)\\
    \kappa_j&:=\id-2\S_j\left(\balpha\circ\balpham\right)+\S_j\left(\balpha\circ\balpham\circ\balpha\circ\balpham\right)\\
    \rho_j&:=-\S_j\left(\balpha\circ\balpham\circ\balpha\right)+\S_j\left(\balpha\right)
\end{align*}
Let $\tilde{\gamma}_j$ (respectively $\tilde{\delta}_j$) be defined by the formula of $\gamma_j$ (resp. $\delta_j$) with $\balpha$ and $\balpham$ exchanged. Then we have
\begin{equation}\label{eq:beta1}
    \beta_{i1}=\left\{\begin{array}{cl}
        2\balpham\circ\alpha-\balpham\circ\balpha\circ\balpham\circ\alpha & \mbox{if $i=1$,} \\
        -\psi_{12}\circ\alpha+\psi_{12}\circ\balpha\circ\balpham\circ\alpha & \mbox{if $i=2$,}\\
        0 & \mbox{otherwise.}
    \end{array}\right.\end{equation}
For even $j$, we have that
\begin{equation}\label{eq:beta2}\beta_{ij}=\left\{\begin{array}{cl}
        \psi_{j-1,j}^{-1}\circ \gamma_j & \mbox{if $i=j-1$,} \\
        \delta_j & \mbox{if $i=j$,}\\
        -\psi_{j, j+1}\circ \gamma_j & \mbox{if $i=j+1$,}\\
        \psi_{j+1, j+2}\circ\psi_{j, j+1}\circ \kappa_j & \mbox{if $i=j+2$,}\\
        0 & \mbox{otherwise.}
    \end{array}\right.
\end{equation}
For odd $j>1$, we have that
\begin{equation}\label{eq:beta3}\beta_{ij}=\left\{\begin{array}{cl}
        \psi_{j-2,j-1}^{-1}\circ\psi_{j-1, j}^{-1}\circ \tilde{\kappa}_j & \mbox{if $i=j-2$,} \\
        \psi_{j-1,j}^{-1}\circ\rho_j & \mbox{if $i=j-1$,}\\
        \tilde{\delta}_j & \mbox{if $i=j$,}\\
        -\psi_{j, j+1}\circ \rho_j & \mbox{if $i=j+1$,}\\
        0 & \mbox{otherwise.}
    \end{array}\right.
\end{equation}
It follows easily from the latter description of $\beta$, that it is the identity on the subspace $V$ of Lemma \ref{lem:identity}. For example, suppose that we want to prove that $\beta|_V^V=\id$ --- proving that $\beta|_V^{V^c}=0$ or $\beta|_{V^c}^V=0$ is done in a similar way. To show that the even columns of $\left(\beta_{ij}|_V^V\right)_{i,j\geq 1}$ are those of the identity, it suffices to show that $\delta_j|_V^V=\id$ and that both $\gamma_j$ and $\kappa_j$ are zero on $V$. Even though $\balpha$ and $\balpham$ aren't necessarily automorphisms, $\balpha \circ \balpham$ and $\balpham\circ\balpha$ are
the identity on the subspace $U$ of Lemma \ref{lem:identity} since $\hsi\left(\alpha^{\pm 1}\right)<\frac{1}{30}$. By Lemma \ref{lem:identity}, $\S_j(\balpha \circ \balpham)$ and $\S_j(\balpham \circ \balpha)$ are the identity on $V$ for all $j$. Thus, we have:
\begin{align*}
    \delta_j|_V^V&=2\left.\S_j\left(\balpha\circ\balpham\right)\right|_V^V-\left.\left[\S_j\left(\balpha\circ\balpham\right)\circ\S_j\left(\balpha\circ\balpham\right)\right]\right|_V^V\\
    &=2\left.\S_j\left(\balpha\circ\balpham\right)\right|_V^V-\left.\S_j\left(\balpha\circ\balpham\right)\right|_V^V\\
    &=\left.\S_j\left(\balpha\circ\balpham\right)\right|_V^V=\id
\end{align*}
For $W=V$ or $W=V^c$ we have:
\begin{align*}
    \gamma_j|_V^W&=-3\left.\left[\S_j\left(\balpham\right)\circ\S_j\left(\balpha\circ\balpham\right)\right]\right|_V^W+2\left.\S_j\left(\balpham\right)\right|_V^W \\
    &+\left.\left[\S_j\left(\balpham\circ\balpha\circ\balpham\right)\circ\S_j\left(\balpha\circ\balpham\right)\right]\right|_V^W \\
    &=-3\left.\S_j\left(\balpham\right)\right|_V^W+2\left.\S_j\left(\balpham\right)\right|_V^W+\left.\S_j\left(\balpham\circ\balpha\circ\balpham\right)\right|_V^W \\
    &=-3\left.\S_j\left(\balpham\right)\right|_V^W+2\left.\S_j\left(\balpham\right)\right|_V^W+\left.\S_j\left(\balpham\right)\right|_V^W=0
\end{align*}
To show that $\gamma_j|_{V^c}^V=0$ one proceeds in a similar way.

\medskip
Put $B:= \left(\frac12,\frac32\right)\times [1, +\infty)$. Let $(x,t)\in B$ and let $(y,s)\not\in B$. We will now show that $\beta_{(x,t)}^{(y,s)}=0$. We have the following two possibilities:
\begin{enumerate}
    \item $(x,t)\in V$.
    \begin{itemize}
        \item If $(y,s)\not\in V$, then $\beta_{(x,t)}^{(y,s)}=0$ because $\beta|_V^{V^c}=0$.
        \item If $(y,s)\in V$, then $\beta_{(x,t)}^{(y,s)}=0$ unless $(x,t)=(y,s)$, because $\beta|_V^V=\id$. But $(x,t)\ne(y,s)$ since we are assuming $(x,t)\in B$ and $(y,s)\not\in B$. Hence $\beta_{(x,t)}^{(y,s)}=0$.
    \end{itemize}
    \item $(x,t)\not\in V$.
    Pick $n\geq 1$ such that $t\in [\tau_n, \tau_{n+1})$. Then
    \[\S(A)_{(x,t)}=\bigoplus_{j=1}^n\S_j(A)_{(x,t)}.\]
    Thus, $\beta_{(x,t)}^{(y,s)}=0$ if and only if $\left(\beta_{ij}\right)_{(x,t)}^{(y,s)}=0$ for all $1\leq j\leq n$ and all $i\geq 1$. Fix $1\leq j\leq n$ and $i\geq 1$. We proceed to show that $\left(\beta_{ij}\right)_{(x,t)}^{(y,s)}=0$. By \eqref{eq:beta1}, \eqref{eq:beta2} and \eqref{eq:beta3}, we know that $\beta_{ij}$ is the composite of an endomorphism $\epsilon:\S_j(A)\to \S_j(A)$ followed by the usual isomorphism $\iota:\S_j(A)\overset{\cong}\to\S_i(A)$. We have that
    \[\left(\beta_{ij}\right)_{(x,t)}^{(y,s)}=\iota_{(z,u)}^{(y,s)}\circ \epsilon_{(x,t)}^{(z,u)},\]
    where $(z,u)$ is the unique point such that $\iota_{(z,u)}^{(y,s)}\ne 0$. It is clear that $(z,u)\not\in B$ since $(y,s)\not\in B$. We will show that $\epsilon_{(x,t)}^{(z,u)}=0$. It follows from \eqref{eq:beta1}, \eqref{eq:beta2} and \eqref{eq:beta3} that $\epsilon$ is a sum of terms, each of which may be the identity or the composite of at most five factors of the form $\S_j\left(\balphapm\right)$---in the case $j=1$ we also have the factor $\alpha$; see \eqref{eq:beta1}. Let $\sigma$ be one of the terms appearing in $\epsilon$. If $\sigma=\id$, then $\sigma_{(x,t)}^{(z,u)}=0$ since $(x,t)\ne (z,u)$. Suppose $\sigma$ is a composite of factors of the form $\S_j\left(\balphapm\right)$, say
    \[\sigma=\varphi_r\circ\cdots \circ\varphi_1\]
    with $1\leq r\leq 5$. If we write $(x_0, t_0):=(x,t)$ and $(x_r,t_r):=(z,u)$, we have:
    \begin{equation}\label{eq:tau}\sigma_{(x,t)}^{(z,u)}=\sum \left(\varphi_r\right)_{(x_{r-1},t_{r-1})}^{(x_r, t_r)}\circ \cdots \circ \left(\varphi_1\right)_{(x_0,t_0)}^{(x_1, t_1)}\end{equation}
    All the points $(x_k,t_k)$ that appear in this sum have $t_k\geq \tau_n-5K$; indeed, this follows from the facts that $t_0=t\geq \tau_n$ and $\vsi(\varphi_k)\leq K$ for all $k$. Since $(x,t)\in V^c\cap B$ and $t\in[\tau_n,\tau_{n+1})$, we have that $x\in\left(\frac12+\frac{1}{6n},\frac32-\frac{1}{6n}\right)$. Since $(z,u)\not\in B$, the latter implies that $|x-z|>\frac{1}{6n}$. Thus, for each term in the sum \eqref{eq:tau} there exists $k$ such that $|x_k-x_{k-1}|>\frac{1}{30n}$. Then each term in \eqref{eq:tau} is zero --- this follows, for $j<n$, from the condition which was
    requirement for the construction of the sequence $\tau_n$ and, for 
    $j=n$, from the fact that $\hsi\left[\S_n\left(\balphapm\right)\right]\leq \frac{1}{30n}$.
    This finishes the proof that $\left(\beta_{ij}\right)_{(x,t)}^{(y,s)}=0$.
\end{enumerate}

In the case of the category $\ODR$, the proof works exactly the same.
The control condition noted in Remark \ref{d-en-z} is crucial for
imitating the same proof.
\end{proof}
As an immediate corollary we get the following theorem.
\begin{thm}\label{vani}
 Let $\alpha:A \to A$ be an $\frac{1}{30}$-automorphism in $\OGR$ or in $\ODR$. Then $\alpha$ has trivial class in $K_1$.
\end{thm}
\begin{proof}
Use the previous proposition together with \ref{lem:K1vanishing}.
\end{proof}

\section{Small matrices in \texorpdfstring{$K_1(R[t,t^{-1}])$}{K1(R[t,t-1])}}\label{sec:squeezing}
Let $\epsilon > 0$. In this section we show that the class of $t$ in $K_1(R[t,t^{-1}])$ can be represented by an $\epsilon$-automorphism. When
$R$ is a regular ring, it follows from this and the proof of \cite{bhs}*{Theorem 2} that any element in $K_1(R[t,t^{-1}])$ has a representative which is an $\epsilon$-automorphism; see \eqref{cuadradito}.

Recall that an element in $K_1(R[t,t^{-1}])$ is the class of an invertible matrix with coefficients in $R[t,t^{-1}]$. We begin by revisiting the notion of \emph{size} in the context of matrices.

Let $R$ be a ring (not neccessarily regular) and $A$  a matrix in $M_n(R[t,t^{-1}])$, then $A$ can be written as
 \[ A=\sum^{i=m}_{i=-m}A_{i}t^{i} \qquad \text{for } A_{i}\in M_n(R).\]
In order to identify $A$ with a morphism in $\widetilde{\cT}^{\ltr}(\R)$ we need to fix a geometric $R$-module $M=(M_x)_{x\in\R}$. Put
\begin{gather*}
M_{x}=\begin{cases*}
        R &if $x\in \operatorname{supp}(M)$, \\
        0 &otherwise 
        \end{cases*} \\
        \text{ with } \operatorname{supp}(M)=\left\{k+\frac{r}{n}: k,r\in \mathbb{Z}\right\}=\ltr \cdot\left\{0,\frac{1}{n},\ldots ,\frac{n-1}{n}\right\}.
\end{gather*} 
 Note that, by $\ltr$-invariance, an endomorphism $\alpha:M\to M$ in $\widetilde{\cT}^{\ltr}(\R)$ is determined by its components $\alpha_x^y$, for $x\in \{0,\frac{1}{n},\ldots ,\frac{n-1}{n}\}$. Let $X=(x_{0},x_{1},\ldots,x_{n-1})\in R^n$ and define
\begin{gather*}
\mathbb{V}_{n}: M_n(R[t,t^{-1}])\rightarrow \widetilde{\cT}^{\langle t\rangle}(\R) \text{ by }\mathbb{V}_{n}(A)=\alpha_{A}: M\rightarrow M \\
\text{where } \alpha_A{(X)}=\sum^{i=m}_{i=-m}A_{i}X^{t}, \text { with } A_{i}X^{t}\in \bigoplus_{j=0}^{j=n-1}M_{i+\frac{j}{n}} = R^n.
\end{gather*}

\begin{defn}\label{size}
For each matrix $A\in M_n(R[t,t^{-1}])$ we define $\operatorname{size}(A)=\size(\mathbb{V}_n(A))$.
\end{defn}

Consider for each $k>0$ the following matrices in $M_n(\R)$: 
\[
D^{k}_{n}=\begin{pmatrix}
k & k-\frac{1}{n} &\ldots & k-\frac{n-2}{n} & k-\frac{n-1}{n}\\
k+\frac{1}{n}  & k &\ldots & k-\frac{n-3}{n} & k-\frac{n-2}{n}\\
\vdots & \vdots & \ddots &\vdots & \vdots\\
k+\frac{n-2}{n} & k+\frac{n-3}{n} &\ldots & k&  k- \frac{1}{n} \\
k+\frac{n-1}{n} & k+\frac{n-2}{n}  &\ldots & k +\frac{1}{n} & k\\
\end{pmatrix}
\]
\[
D^{0}_{n}=\begin{pmatrix}
0 & \frac{1}{n} &\ldots & \frac{n-2}{n} & \frac{n-1}{n}\\
\frac{1}{n}  & 0 &\ldots & \frac{n-3}{n} & \frac{n-2}{n}\\
\vdots & \vdots & \ddots &\vdots & \vdots\\
\frac{n-2}{n} & \frac{n-3}{n} &\ldots & 0&   \frac{1}{n} \\
\frac{n-1}{n} & \frac{n-2}{n}  &\ldots &  \frac{1}{n} & 0\\
\end{pmatrix}
\text{ and }
D^{-k}_{n}= (D^{k}_{n})^{t}
\]
Then, we have that
\[
\size(A_{k})=\max\{d^{k}_{ij}: a_{ij}\neq 0\}\text{ and }
\size{(A)}=\max_{-m\leq k \leq m}{\{ \size{(A_{k})}\}}.
\]

\begin{rem}\label{remgl}
The size of a matrix depends on the dimension $n$, so it is not invariant with respect to stabilization in $\operatorname{GL}_{n}(R[t,t^{-1}])$. Considering $B\in \operatorname{GL}_{n}(R)$ we have
$$\size{\left( \begin{array}{cc}B & 0\\ 0& I_{m} \end{array}\right) }= \frac{n}{n+m}\size{(B)}$$
\end{rem}

\begin{prop}\label{tchiq}

Let $x\in K_{1}(R[t,t^{-1}])$ be the class of $t\in \operatorname{GL}_{1}(R[t,t^{-1}])$ and $\epsilon > 0$. If $\mathbb{U}$ is the functor of Lemma \ref{lem equiv}, then there exists an $\epsilon$-automorphism $\alpha \in \tilde{\cT}^{\langle t \rangle}(\R)$ such that $[\mathbb{U}(\alpha)]= x$.

\end{prop}

\begin{proof}
Take $n\in \mathbb{N}$ such that $\frac{1}{n}<\epsilon$. Let $C_{n}\subset{\mathbb{R}}$ be the $\mathbb{Z}$-compact set 
$$
C_{n}= \ltr \cdot \left\{\frac{i}{n}: i\in\{ 0,\ldots, n-1\}\right\}= \left\{m+\frac{i}{n}: m\in \mathbb{Z},  i\in \{ 0,\ldots, n-1\} \right\}
$$
Define for $i\in\{0,\ldots n-1\}$ the following geometric module:
$$
Q^{i}[n]=\{Q^{i}[n]_{y}\}_{y\in \mathbb{R}} \quad\text{with }
Q^{i}[n]_{y}=\left\{\begin{array}{ll}  R & \mbox{ if  $y=m+\frac{i}{n}$ with $m\in\mathbb{Z}$ }\\
0&\mbox{ otherwise.}
  \end{array}\right. 
$$
If we note $Q^{0}[n]=P$, we have that
$$
P=\{P_{y}\}_{y\in \mathbb{R}}, \quad P_{y}=\left\{\begin{array}{ll}  R & \mbox{ if  $y\in \Z$ }\\
0&\mbox{ otherwise }
\end{array}\right. 
$$
and that $Q^{i}[n]$ is a translation of $P$ by $\frac{i}{n}$.
As such, there is an isomorphism $\delta_{i}:P\rightarrow Q^{i}[n]$.

Define $\gamma:P\rightarrow P$ as the automorphism such that $\gamma_w^y:P_{w}\rightarrow P_{y}$ is $\id_{R}$ when $y=w+1$ and the null map otherwise. Note that $\size(\gamma)=1$. Consider now
$\gamma_{ij}: Q^{i}[n]\rightarrow Q^{j}[n]$, given by $\gamma_{ij}
=\delta_{j}\circ \gamma \circ \delta^{-1}_{i}$. Abusing notation
we write $\gamma_{ij}$ as $\gamma$.

Define $Q[n]=\displaystyle\bigoplus^{n-1}_{i=0}Q^{i}[n]$ and $\xi=\{\xi_{ij}\}_{i,j\in \{0,\ldots,n-1\}}:Q[n]\rightarrow Q[n]$ given by
$$
\xi_{ij}:Q^{j}[n] \rightarrow Q^{i}[n]  =\left\{ \begin{array}{ll}
\id & \mbox{ if $j=i-1$, $j\geq 1$}\\
\gamma & \mbox{ if $j=n-1$, $i=0$}\\
(-1)^{n}\id & \mbox{ if $j=0$, $i=1$}\\
0 & \mbox{ otherwise.}\\
\end{array}\right.
$$
Note that $\supp(Q[n])=C_{n}$, $\size{(\xi)}=\frac{1}{n}$ and  $\xi$ is represented by the matrix
$$
\left( \begin{array}{llllll} 
0 & 0 &0&\ldots & 0 &\gamma \\
(-1)^{n+1}\id & 0&0&\ldots & 0 & 0\\
0 & \id &0 &\ldots & 0 & 0\\
\vdots & \vdots &\vdots & \ddots &\vdots & \vdots\\
0 & 0 &0&\ldots & 0 & 0\\
0 & 0  &0 &\ldots & \id & 0\\
\end{array}\right).
$$
Is easy to see that $\xi$ is an automorphism with $\xi^{-1}$  represented by the matrix
$$
\left( \begin{array}{llllll} 
0 & (-1)^{n+1}\id &0&\ldots & 0 &0\\
0 & 0&\id &\ldots & 0 & 0\\
0 & 0 &0 &\ldots & 0 & 0\\
\vdots & \vdots &\vdots & \ddots &\vdots & \vdots\\
0 & 0 &0&\ldots & 0 & \id\\
\gamma^{-1} & 0  &0 &\ldots & 0 & 0 \\
\end{array}\right)
$$
and that $\size(\xi^{-1})=\frac{1}{n}$.
Letting $\nu:Q[n]\rightarrow Q[n]$ be the automorphism represented by the matrix
$$
\left( \begin{array}{lllll} 
\gamma & 0 &\ldots & 0 &0\\
0 & \id &\ldots & 0 & 0\\
\vdots & \vdots & \ddots &\vdots & \vdots\\
0 & 0 &\ldots & \id & 0\\
0 & 0  &\ldots & 0 & \id\\
\end{array}\right), 
$$
we have that
$$
[\mathbb{U}(\nu)]=[\mathbb{U}(\xi)]=x.
$$
\end{proof}

\begin{prop}\label{Achic}
Let $R$ be a regular ring. For all $x\in K_{1}(R[t,t^{-1}])$ and $\epsilon > 0$ there exists $\epsilon$-automorphism $\alpha \in \widetilde{\cT}^{\langle t \rangle}(\R)$ such that $[\mathbb{U}(\alpha)]= x$.
\end{prop}
\begin{proof}
Every $x\in K_{1}(R[t,t^{-1}])$ is represented by $t^{n}MN$ with $n\in \mathbb{Z}$, $M\in \operatorname{GL}_{m}(R)$  and $[N]=[\id]$ (because $R$ is regular), see \cite{bhs}. Then $$x=[t^n]+[M].$$
By Proposition \ref{tchiq} we can consider $[t^{n}]=[\mathbb{U}(\alpha_{n})]$ such that $$\size(\alpha_{n})=\size(\alpha_{n}^{-1})<\frac{\epsilon}{2}$$ and by Remark \ref{remgl} there exists $\alpha_{M}$  such that $[\mathbb{U}(\alpha_{M})]= [M]$ and $$\size(\alpha_{M})=\size(\alpha^{-1}_{M})<\frac{\epsilon}{2}.$$
Then taking $\alpha=\alpha_{n}\circ \alpha_{M}$ we have
$$[\mathbb{U}(\alpha)]=x \qquad \qquad \size{(\alpha)}<\epsilon \qquad \size{(\alpha)^{-1}}<\epsilon.$$
\end{proof}

\begin{cor}\label{assepi} For $R$ a regular ring, the assembly map
\[K_2(\widetilde{\cD}^{\ltr}(\R)) \to K_1(\widetilde{\cT}^{\ltr}(\R))\]
is an epimorphism. 
\end{cor}

\begin{proof} By Proposition \ref{Achic} and Corollary \ref{vani}, the image of every element of $K_1(\widetilde{\cT}^{\langle t \rangle}(\R))$ in $K_1(\widetilde{\cO}^{\langle t \rangle}(\R))$ is trivial, hence by \eqref{ladder} the assembly is surjective.
\end{proof}

\end{document}